\documentclass[10pt]{article}
\usepackage{amssymb}
\usepackage{amsmath}
\usepackage{amsthm}
\usepackage{mathabx} 
\usepackage[backref=page,colorlinks=true]{hyperref} 
\usepackage[capitalise]{cleveref} 
\usepackage{graphicx}
\usepackage{psfrag}
\usepackage[small]{caption}
\usepackage{xcolor}

\makeatletter

\def\myMRbibitem{\@ifnextchar[\my@lbibitem\my@bibitem}

\def\mybiblabel#1#2{\@biblabel{{\hyperref{http://www.ams.org/mathscinet-getitem?mr=#1}{}{}{#2}}}}

\def\myhyperanchor#1{\Hy@raisedlink{\hyper@anchorstart{cite.#1}\hyper@anchorend}}

\def\my@lbibitem[#1]#2#3#4\par{%
    \item[\mybiblabel{#2}{#1}\myhyperanchor{#3}\hfill]#4%
    \@ifundefined{ifbackrefparscan}{}{\BR@backref{#3}}%
    \if@filesw{\let\protect\noexpand\immediate
       \write\@auxout{\string\bibcite{#3}{#1}}}\fi\ignorespaces%
}

\def\my@bibitem#1#2#3\par{%
    \refstepcounter\@listctr
    \item[\mybiblabel{#1}{\the\value\@listctr}\myhyperanchor{#2}\hfill]#3%
    \@ifundefined{ifbackrefparscan}{}{\BR@backref{#2}}%
    \if@filesw\immediate\write\@auxout
        {\string\bibcite{#2}{\the\value\@listctr}}\fi\ignorespaces%
}

\makeatother

\makeatletter
\newcommand{\subjclass}[2][2010]{%
  \let\@oldtitle\@title%
  \gdef\@title{\@oldtitle\footnotetext{#1 \emph{Mathematics subject classification.} #2}}%
}
\newcommand{\keywords}[1]{%
  \let\@@oldtitle\@title%
  \gdef\@title{\@@oldtitle\footnotetext{\emph{Key words and phrases.} #1.}}%
}
\makeatother

\theoremstyle{plain}
	\newtheorem{mainthm}{Theorem} 
	
	\crefname{mainthm}{Theorem}{Theorems} 
	\newtheorem{theorem}{Theorem}[section]
	\newtheorem{lemma}[theorem]{Lemma}
	
	\newtheorem{proposition}[theorem]{Proposition}
	\newtheorem*{claim}{Claim}
\theoremstyle{remark}
	\newtheorem{remark}[theorem]{Remark}
	\newtheorem{example}[theorem]{Example}
	\newtheorem*{ack}{Acknowledgements}

\newcommand{\closeremark}{\hfill$\triangleleft$}  

\crefname{subsection}{\S}{\S\S} 
\Crefname{subsection}{\S}{\S\S} 
\crefname{subsubsection}{\S}{\S\S} 
\Crefname{subsubsection}{\S}{\S\S} 

\numberwithin{equation}{section}

\renewcommand{\epsilon}{\varepsilon}
\renewcommand{\phi}{\varphi}
\renewcommand{\setminus}{\smallsetminus}

\newcommand{\R}{\mathbb R}

\newcommand{\Z}{\mathbb Z}
\newcommand{\C}{\mathbb C}

\newcommand{\GL}{\mathrm{GL}}

\newcommand{\Isom}{\mathrm{Isom}}
\newcommand{\Kill}{\mathrm{Kill}}

\newcommand{\id}{\mathrm{id}}  
\newcommand{\Id}{\mathit{Id}}  
\DeclareMathOperator\bary{bar}
\newcommand{\dist}{d}

\newcommand{\fm}{\mathfrak{m}}
\DeclareMathOperator\displ{displ}
\DeclareMathOperator\drift{drift}
\DeclareMathOperator\speed{speed}
\newcommand{\triang}{\bigtriangleup}
\DeclareMathOperator\grad{grad}
\DeclareMathOperator\supp{supp}
\renewcommand{\mid}{\mathop{\mathrm{mid}}}
\newcommand{\bysame}{\makebox[3em]{\hrulefill}\thinspace. }

\newcommand{\arxiv}[1]{Preprint \href{http://arxiv.org/abs/#1}{arXiv:{#1}}}

\renewcommand*{\backref}[1]{}
\renewcommand*{\backrefalt}[4]{\quad \tiny 
    \ifcase #1 (Not cited.)%
    \or        (Cited on page~#2.)%
    \else      (Cited on pages~#2.)%
    \fi}

\begin{document}

\author{Jairo Bochi\footnote{Partially supported by CNPq (Brazil), FAPERJ (Brazil), and the Math-AMSUD Project DySET.} \quad \& \quad Andr\'es Navas\footnote{Partially supported by Fondecyt Project 1120131 and PIA-Conicyt Project ACT 1103 DySyRF.}}

\title{A geometric path from zero Lyapunov exponents to rotation cocycles}

\date{May, 2013}


\maketitle

\thispagestyle{empty} 

\begin{abstract}
We consider cocycles of isometries on spaces of nonpositive curvature~$H$. We show that the supremum of the drift over all invariant ergodic probability measures equals the infimum of the displacements of continuous sections under the cocycle dynamics. In particular, if a cocycle has uniform sublinear drift, then there are almost invariant sections, that is, sections that move arbitrarily little under the cocycle dynamics. If, in addition, $H$ is a symmetric space, then we show that almost invariant sections can be made invariant by perturbing the cocycle. 
\end{abstract}


\section{Introduction} \label{s.intro}

\subsection{Basic setting and results}\label{ss.general_D}

Let $F \!: \Omega \to \Omega$ 
be a continuous mapping of a compact Hausdorff topological space $\Omega$.
A \emph{cocycle} over the dynamics $F$ is a continuous function   
\begin{equation}\label{e.discrete_cocycle}
A \!: \Omega \to G,
\end{equation}
where $G$ is a topological group. 
We denote $A^{(0)}(\omega) := \id$ and for $n \in \Z_+$,
\begin{equation}
\label{e.products}
A^{(n)}(\omega) := A(F^{n-1} \omega) \cdots A(F \omega) A(\omega).
\end{equation}
Notice the \emph{cocycle relation}
\begin{equation}\label{e.cocycle}
A^{(n+m)}(\omega) = A^{(m)}(F^n \omega) A^{(n)}(\omega).
\end{equation}

Two cocycles $A$ and $B$ over $F$ are said to be {\em cohomologous} 
whenever there exists a continuous map $U\!: \Omega \to G$ such that
\begin{equation}\label{e.cohomologous}
A(\omega) = U(F \omega) B(\omega) U(\omega)^{-1}, 
\quad \text{for all } \omega \in \Omega.
\end{equation}

\medskip

In most of this paper, $G$ will be the group $\Isom(H)$
of isometries of a metric space $(H,d)$;
then $A$ is called a \emph{cocycle of isometries}.
We will assume at least that $H$ is a Busemann space (i.e., a separable complete geodesic space of nonpositive curvature in the sense of Busemann). 
The group $\Isom(H)$ is endowed 
with the bounded--open topology. (Definitions are given in \cref{sss.Busemann}.) 

\medskip

The \emph{maximal drift} of a cocycle of isometries is defined as 
\begin{equation}\label{e.drift_D}
\drift(F,A)  := \lim_{n\to \infty} \frac{1}{n} 
\sup_{\omega \in \Omega} \dist \big( A^{(n)}(\omega) p_0, p_0 \big).
\end{equation}
Notice that the limit exists 
by subadditivity, and is independent of the choice of $p_0 \in H$.

\begin{remark}\label{r.furman}
It follows from Kingman's subadditive ergodic theorem that 
for every ergodic probability measure $\mu$ for $F \!:\Omega\to \Omega$, the limit 
$$
\lim_{n\to \infty} \frac{1}{n} \dist \big( A^{(n)}(\omega) p_0, p_0 \big)
$$
exists for $\mu$-almost every $\omega \in \Omega$, and is a constant 
(also independent of $p_0 \in H$).
This is the \emph{drift} of the cocycle of isometries with respect to the measure~$\mu$;
let us denote it by $\drift(F,A,\mu)$.\footnote{Let us mention that the results of \cite{KarM} give important information in the case $\drift(F,A,\mu)>0$.}
The following ``variational principle'' 
holds\footnote{This follows from \cite[Thrm.~1]{Schreiber} or \cite[Thrm.~1.7]{SS}. 
Although these references assume $\Omega$ to be compact metrizable, the proofs also 
work for compact Hausdorff $\Omega$. (See also the proof of Proposition~1 in 
\cite{AB}.)}:
\begin{equation}\label{e.variational}
\drift(F,A) = \sup_\mu \drift(F,A,\mu),
\end{equation}
where $\mu$ runs over all invariant ergodic probabilities for $F$.~\closeremark 
\end{remark}

We say that the cocycle has \emph{uniform sublinear drift} if $\drift(F,A) = 0$. 
By the remark above, this happens if and only if $A$ has zero drift with 
respect to every $F$-invariant probability measure.

\medskip

The \emph{displacement} of a (continuous) section $\phi \!: \Omega \to H$  
is defined by
\begin{equation}\label{e.displacement}
\displ(\phi) := 
\sup_{\omega \in \Omega}
\dist \big( A(\omega) \phi(\omega), \phi(F \omega) \big) \, .
\end{equation}
(When necessary, we use the more precise notation $\displ_{F,A}(\phi)$.)

Notice that $\displ(\phi) = 0$ if and only if the section $\phi$ is 
\emph{invariant}, that is, $A(\omega) \phi(\omega) = \phi(F \omega)$ 
holds for every $\omega \in \Omega$.

It is not hard to show (see \cref{ss.mindispl_D}) 
that the displacement of any continuous section $\phi \!: \Omega \to H$
is at least the drift of the cocycle:
\begin{equation} \label{e.mindispl_D}
\displ(\phi) \ge \drift(F,A). 
\end{equation}
Our first main result is a converse of this fact:

\begin{mainthm}[Existence of sections of nearly minimal displacement; discrete time] \label{t.slow_D}
Assume that $H$ is a Busemann space.
Given a cocycle $A \!: \Omega \to \Isom(H)$ over 
$F \!: \Omega \to \Omega$, for each $\epsilon>0$ there exists a continuous 
section $\phi \!:\Omega \to H$ such that $\displ(\phi) \le \drift(F,A) + 
\epsilon$.
\end{mainthm}

Together with \eqref{e.mindispl_D}, this theorem implies that the maximal 
drift is the infimum of the displacements of the continuous sections $\phi$. 

\bigskip

For the next result, we need extra assumptions on the space $H$,
the most important being that $H$ is a symmetric space (see \cref{sss.symm}).


\begin{mainthm}[Creating invariant sections; discrete time]\label{t.closing_D}
Assume that $H$ is 
\begin{itemize}
\item either a proper\footnote{A metric space is 
called \emph{proper} if bounded closed sets are compact.} Busemann space;
\item or a space of bounded nonpositive curvature in the sense of 
Alexandrov\footnote{See \cref{sss.alexandrov} for the definitions.}.
\end{itemize}
Also assume that $H$ is symmetric. 
Let $A$ be a cocycle of isometries of $H$ with uniform 
sublinear drift.
Then there exists a cocycle $\tilde A$ arbitrarily close to $A$,
that has a continuous invariant section, and so is cohomologous to a cocycle taking values in the stabilizer in $\Isom (H)$ of a point $p_0 \in H$.
\end{mainthm}

Here, the approximation is meant in the following sense:
There is a sequence of cocycles $\tilde A_N$ satisfying the conclusions of the theorem,
and such that for every bounded set $B$ of $H$, the sequence $\tilde A_N(\omega) p$
converges uniformly to $A(\omega) p$ uniformly with respect to $(\omega,p) \in \Omega \times B$.

\medskip

To prove \cref{t.slow_D}, we explicitly construct sections~$\phi$ that almost realize the drift. The main construction uses an appropriate concept of barycenter. This construction is suitable for extensions to flows and nilpotent group actions, as we will see. We also give an alternative argument (based on the referee's comments); however, this argument seems to be less suitable for generalizations. 

The proof of \cref{t.closing_D} is also explicit: 
we use the symmetries of $H$ to construct the required perturbation.
Some care is needed to ensure that the perturbation is small.
While the task is easy in the locally compact case, the general case 
requires finer geometric arguments, making full use of the assumptions on curvature.
In any case, what we ultimately show is that the space $H$ has a certain \emph{uniform homogeneity}
property, which may be of independent interest.

\medskip
 
For another interpretation of \cref{t.closing_D}, see \cref{r.drift_as_TL}.
For some extensions of the results above, see \cref{r.family,r.non-perturbative,r.access,r.fibered_D}.

\subsection{Examples and applications} \label{ss.examples}

The simplest space $H$ to which our results apply is the real line.
If we restrict ourselves to orientation-preserving isometries,
then \cref{t.slow_D,t.closing_D} become results on $\R$-valued cocycles.
They appear repeatedly in the literature; 
see \cite{MOP1}, \cite[Prop.~2.13]{Kat}, \cite[Prop.~6]{CNP2}. 
We call this the \emph{classical} setting.

\medskip

Another natural situation is when $H$ is the euclidean space $\R^n$. 
For an interesting class of examples of cocycles of isometries of $\R^2$ that have 
uniform sublinear drift, see \cite[\S~2.3]{CNP2}.

\medskip

Other situations where our results can be applied
are naturally related to \emph{matrix cocycles}, as we now explain.

We say that a matrix cocycle $A \!: \Omega \to \GL(d,\R)$
has \emph{uniform subexponential growth} if 
$$
\lim_{n \to +\infty} \frac{1}{n} 
\log \left\| (A^{(n)}(\omega))^{\pm 1} \right\| = 0
\quad \text{uniformly over $\omega \in \Omega$,}
$$
for some (and hence any) matrix norm $\|\mathord{\cdot} \|$.

We will show the following:

\begin{theorem}\label{t.matrix}
Let $G$ be an algebraic subgroup of $\GL(d,\R)$ that is closed under matrix transposition,
and let $K$ be its intersection with the orthogonal group $\mathord{O}(n)$.
Let $A \!: \Omega \to G$ be a cocycle with uniform subexponential growth. 
Then there exists a cocycle $\tilde A \!: \Omega \to G$
arbitrarily close to $A$ that is cohomologous 
to a cocycle taking values in $K$.
\end{theorem}

For an elementary proof of this result in the case $G = \GL(d,\R)$
(assuming $F$ invertible), and for further applications, see the companion paper \cite{BN_elementary}.

Other examples of groups $G$ where \cref{t.matrix}
applies are the complex general linear group $\GL(n,\C)$ (embedded in $\GL(2n,\R)$ in the usual way)
and the symplectic group $\mathrm{Sp}(2n)$; in both examples, $K$ is the unitary group $\mathrm{U}(n)$.

\Cref{t.matrix} essentially follows from \cref{t.closing_D}
applied to the space $H:=G/K$.
We can also obtain similar results for infinite-dimensional Lie groups.
See \cref{ss.matrix} for details.

\subsection{Continuous time versions}\label{ss.general_C}

Now we assume that $H$ is a Cartan--Hadamard manifold, that is, a complete simply 
connected Riemannian manifold (possibly of infinite dimension, modeled on a 
Hilbert space) of nonpositive sectional curvature. (See \cref{sss.killing} 
for details.)

A \emph{semiflow} $\{F^t\}$ on $\Omega$ is a continuous-time dynamical system,
that is, a continuous map $(\omega,t) \in \Omega \times \R_+ \mapsto F^t \omega \in \Omega$
such that
$F^0 =\id$ and $F^{t+s}= F^t \circ F^s$ for all $s,t$ in $\R$.

A \emph{cocycle of isometries} (of $H$) over $\{F^t \}$ is a 1-parameter 
family of maps $A^{(t)} \!: \Omega \to \Isom(H)$ (where $t\in \R_+$) satisfying 
\begin{equation}\label{e.cocycle_C}
A^{(0)}(\omega) = \id, \quad A^{(s+t)}(\omega) = A^{(s)}(F^t \omega) A^{(t)}(\omega)
\end{equation}
and such that 
\begin{equation}\label{e.regularity}
\text{$A^{(t)}(\omega) p$ is continuous w.r.t.\ $(t, \omega, p)$ 
and continuously differentiable w.r.t.\ $t$.}
\end{equation} 

Given the cocycle $\{ A^{(t)} \}$, we can associate the vector field $\mathfrak{a}(\omega)$ on $H$ defined by  
\begin{equation}\label{e.generator}
\mathfrak{a} (\omega)(p) = \left. \frac{\partial}{\partial t} A^{(t)}(\omega) p \right|_{t=0} \, .
\end{equation}
This defines a continuous map $\mathfrak{a} \!: \Omega \to \Kill(H)$,
where $\Kill(H)$ is the set of Killing fields on $H$. 
(See \cref{sss.killing} for the topology on $\Kill(H)$.) 
Conversely, given any continuous map $\mathfrak{a} \!: \Omega \to \Kill(H)$,
there is a unique cocycle of isometries $(F^t, A^{(t)})$ 
satisfying the ODE 
\begin{equation}\label{e.ODE}
\frac{\partial}{\partial t} A^{(t)}(\omega) p 
= \mathfrak{a} (F^t \omega) (A^{(t)}(\omega) p) \, .
\end{equation}
The map $\mathfrak{a}$ is called the \emph{infinitesimal generator} of the cocycle.

We define the \emph{maximal drift} of a continuous-time cocycle 
of isometries in the same way as the discrete-time situation: 
\begin{equation}\label{e.drift_C}
\drift(F,A)  := \lim_{t\to \infty} \frac{1}{t} \sup_{\omega \in \Omega} 
\dist \big( A^{(t)}(\omega) p_0, p_0 \big), \quad \text{$p_0 \in H$ arbitrary.}
\end{equation}
Again, we say that the cocycle has \emph{uniform sublinear drift} if $\drift(F,A) = 0$.

A (continuous) section $\phi \!: \Omega \to H$ is said to be 
\emph{differentiable with respect to the semiflow}\footnote{A 
similar definition appears in \cite{Schw}.} if
for every $\omega \in \Omega$, the derivative 
\begin{equation}\label{e.d_wrt_semiflow}
\phi'(\omega) := \left. \frac{\partial}{\partial t} \phi(F^t \omega) \right|_{t=0}
\end{equation}
exists and defines a continuous map $\phi'  \!: \Omega \to TH$. 
Is $\phi$ is such a section then its \emph{speed} with respect to $\{ A^t \}$ is defined by:
\begin{equation}\label{e.speed}
\speed(\phi) := 
\sup_{\omega \in \Omega} \big\| \mathfrak{a} (\omega)(\phi(\omega)) - \phi'(\omega) \big\| \, ,
\end{equation}
where $\mathfrak{a}$ is the infinitesimal generator of the cocycle.
(A more precise notation is $\speed_{F,\mathfrak{a}}(\phi)$.)
The speed is the continuous-time analogue of the displacement \eqref{e.displacement}.
Notice that $\speed(\phi) = 0$ holds if and only if the section $\phi$ is 
\emph{invariant}, that is, $A^{(t)}(\omega) \phi(\omega) = \phi(F^t \omega)$.

Analogously to \eqref{e.mindispl_D}, we have
\begin{equation}\label{e.minidispl_C}
\speed(\phi) \ge \drift(F,A).
\end{equation}

The continuous-time versions of \cref{t.slow_D,t.closing_D} are given below:

\begin{mainthm}[Existence of sections of nearly minimal speed; continuous time]
\label{t.slow_C}
Assume that $H$ is a Cartan--Hadamard manifold.
Given a continuous-time cocycle of isometries $\{ A^{(t)} \}$ of $H$ over a semiflow 
$\{F^t\}$ on $\Omega$, 
for each $\epsilon>0$ there exists a continuous section $\phi \!:\Omega \to H$
that is differentiable with respect to the semiflow $\{F^t\}$ and 
such that $\speed(\phi) \le \drift(F,A) + \epsilon$.
\end{mainthm}

\medskip

\begin{mainthm}[Creating invariant sections; continuous time]\label{t.closing_C}
In the context of the preceding theorem, assume moreover 
that $H$ is a symmetric space
and that the cocycle $\{A^{(t)}\}$ has uniform sublinear drift. 
Let  $\mathfrak{a}$ be the infinitesimal generator of  $\{A^{(t)}\}$.
Then there exists $\tilde{\mathfrak{a}} \!: \Omega \to \Kill(H)$,
arbitrarily close to $\mathfrak{a}$ such that the 
associated cocycle $\{ \tilde A^{(t)} \}$ has an invariant continuous section 
(and hence is cohomologous to a cocycle taking values in the stabilizer 
of a point $p_0 \in H$).
\end{mainthm}

Here, the convergence of a sequence $\tilde{\mathfrak{a}}_N  \!: \Omega \to \Kill(H)$
to $\mathfrak{a}$ is in the following sense: for each bounded subset $B$ of $H$, the 
sequence $\tilde{\mathfrak{a}}_N (\omega)(p)$ converges to $\mathfrak{a}(\omega)(p)$ 
uniformly with respect to $(\omega,p) \in \Omega \times B$. (See 
\cref{sss.killing} for more precise explanation about the topologies.)

\medskip

Although the proofs of these theorems follow the same ideas as the discrete-time versions,
the technical details are of a different nature. 
Thus we prove the two kinds of results in a nearly independent way.

As its discrete-time analogue, the proof of \cref{t.slow_C} uses barycenters,
but we also need to concern ourselves with differentiability with respect to the flow.

The proof of \cref{t.closing_C} uses a infinitesimal uniform homogeneity property of the space $H$.
The proof of this property, like of its macroscopic version, 
is simpler in the locally compact case but uses finer geometrical arguments in the general case.

\subsection{Other group actions}\label{ss.general_nilp}

It is natural to ask whether the previous theorems extend to cocycles of isometries over actions 
of  semigroups more complicated than $\Z_+$ or $\R_+$. We will concentrate on 
discrete groups, leaving the generalizations to continuous groups as a task to the reader.

Given a group $\Gamma$ acting on the left by homeomorphisms of a compact Hausdorff topological space 
$\Omega$ and a topological group $G$, a cocycle over the $\Gamma$-action with values 
in $G$ is a continuous map 
\begin{align*}
	A \!: \Gamma \times \Omega &\to G \\ 
	(g,\omega) &\mapsto A^{(g)}(\omega)
\end{align*}
such that 
$$
A^{(gh)} (\omega) = A^{(g)} (h \omega) A^{(h)} (\omega)
\quad\text{for all $g$,$ h$ in $\Gamma$ and all $\omega \in \Omega$.}
$$

In the classical case (that is, when $G = \mathbb{R}$), 
the analogue of \cref{t.slow_D} holds for nilpotent group actions,
but it does not hold for solvable (in particular, amenable) group actions; see \cite{MOP2,MOP1}. 
The next result establishes that the analogue of \cref{t.slow_D} remains true for cocycles of isometries 
over abelian group actions. Generalizations for virtually nilpotent 
group actions will be discussed in \cref{s.proofs_G}. 

Assume that $A$ is a cocycle of isometries of a space $(H,d)$.
We say that $A$ has \emph{sublinear drift along cyclic 
subgroups} if for each fixed $p_0 \in H$ and all $g \in \Gamma$,  the limit
$$
\lim_{n \to \infty} \frac{1}{n} d (A^{(g^n)} (\omega) p_0, p_0)
$$
equals zero uniformly on $\Omega$. 
We say that $A$ \emph{admits almost-invariant sections} if there 
exists a sequence of continuous functions $\varphi_N: \Omega \to H$ such that for all 
$g \in \Gamma$,
$$\lim_{N \to \infty} d \big( A^{(g)} (\omega) \varphi_N (\omega), \varphi_N (g \omega)  \big) = 0$$
uniformly on $\Omega$.

\begin{mainthm}[Existence of almost-invariant sections; abelian groups]\label{t.slow_G}
Let $\Gamma$ be a finitely generated abelian group acting by homeomorphisms of a 
compact Hausdorff metric space $\Omega$. 
Let $A$ be a cocycle over this group action 
with values in the group of isometries of a Busemann space $H$. 
If $A$ has sublinear drift along cyclic subgroups,  
then $A$ admits almost-invariant sections.
\end{mainthm}

As it is easy to see, in order to check the condition on drift along all cyclic subgroups above, 
it suffices to check it only for those associated to the generators of the group $\Gamma$. 

\medskip

%
%
%

We actually provide an extension of \cref{t.slow_G} to virtually nilpotent groups $\Gamma$; 
see \cref{s.proofs_G}. This is 
somewhat related to results obtained by de~Cornulier, Tessera and Valette \cite{CTV}; 
in their considerations, $\Omega$ is a point and $H$ is a Hilbert space.

\subsection{Further questions}

We next mention a few other questions that are suggested by our results.

\medskip

The first question is whether our results can be extended to cocycles of \emph{semicontractions}.
Notice that the basic fact \eqref{e.mindispl_D} 
(as well as the related theorem from \cite{KarM}) still hold in this case.

\medskip

Fix the dynamics $F \!: \Omega\to\Omega$, and let $A \!: \Omega \to \R$ be a real function.
As mentioned above (\cref{ss.examples}), we can regard this as a cocycle of orientation-preserving 
isometries of the line. It follows from the ``variational principle'' \eqref{e.variational}
that $\drift(F, A) = \sup_\mu \left| \int A \, \mathrm{d}\mu \right|$.
The study of the probability measures which realize this supremum
is an interesting problem with rich ramifications in {\em ergodic optimization}; see \cite{Jenkinson}. 
In this spirit, it could be also interesting to study \emph{drift maximizing measures}, i.e.,
those which realize the supremum in \eqref{e.variational}.

A perhaps related problem is to find out when a displacement minimizing section exists in \cref{t.slow_D} 
(or, in the language of \cref{r.drift_as_TL}, when is $\Gamma$ a semi-simple isometry.)

\medskip

All our results concern approximations in the $C^0$ topology.
It is natural to ask whether the differentiability class can be improved.
However, this is non-trivial already in the classical case (i.e., with $G=\R$),
where it is closely related to the existence of invariant distributions
for the base dynamics (see \cite{Kat,AK,NT}).
We do not know whether such a relation can be extended to the cocycles considered in this work.

\medskip

Finally, a natural problem raised by this work concerns the case of diffeomorphisms:
in the situation of \cref{t.matrix}, if a cocycle is given by the derivative of a diffeomorphism, 
then it is natural to require that the perturbed cocycle is also a derivative.
More precisely, we pose the following question:
given a diffeomorphism $f$ of a compact manifold all whose Lyapunov exponents are zero,
under what circumstances is $f$ close to a diffeomorphism that is conjugate to an isometry?
Our methods fail in dealing with this problem, 
as it involves a simultaneous (and coherent) perturbation of the base dynamics and the cocycle.
Let us point out, however, that the answer (in $C^1$ regularity) is known in the one-dimensional case;  
see \cite{BoGu,Navas-rapprochements}.

\subsection{Organization of the paper}

In \cref{s.proofs_D} we prove the discrete-time \cref{t.slow_D,t.closing_D};
we also explain how to obtain the application \cref{t.matrix}.

In \cref{s.proofs_G} we deal with other (still discrete) group actions,
thus proving \cref{t.slow_G} and an extension of it.
That section is shorter and uses the tools explained in the previous one.

In \cref{s.proofs_C} we deal with continuous-time cocycles, 
thus proving \cref{t.slow_C,t.closing_C}.
Although the main ideas of the proofs are similar to those of the discrete-time results, 
the technical details are somewhat different,
and this section is actually nearly independent from the previous ones.

Along the way, we explain the geometrical tools and properties that are required for 
the proofs. Some of these properties become simpler to prove if the space $H$ 
is proper (i.e., finite-dimensional in the case of manifolds).
Thus in \cref{s.proofs_D,s.proofs_C} we give proofs for the proper case,
and leave for Appendices~\ref{s.alexandrov} and \ref{s.killing} 
the geometrical arguments which allow to extend the proofs to the general case.
A technical property which is needed for the proof of \cref{t.closing_C}
is proved in Appendix~\ref{s.diff_cartan}.

\bigskip

\begin{ack}
We thank A.~Avila, E.~Garibaldi, N.~Gourmelon, A.~Karlsson, A.~Kocsard, 
G.~Larotonda, M.~Ponce and R.~Tessera for valuable discussions. 
We thank the hospitality of the Institut de Math\'ematiques de Bourgogne, where this work started.
We are grateful to the referee for remarks and suggestions that helped to improve the paper.

J.~Bochi was partially funded by the research projects CNPq 309063/2009-4 and FAPERJ 103.240/2011.
A.~Navas was partially funded by the research projects ACT-1103 DySyRF and FONDECYT 1120131.
\end{ack}

\section{Discrete-time cocycles} \label{s.proofs_D}

In this section we prove \cref{t.slow_D,t.closing_D}.
We also explain how to deduce the application \cref{t.matrix}.

\subsection{The easy inequality}\label{ss.mindispl_D}

Let us prove that the displacement of any section is an upper bound 
for the drift of the cocycle, as asserted in \eqref{e.mindispl_D}.
Here, no assumptions on the geometry of $H$ are needed.

\begin{proof}[Proof of \eqref{e.mindispl_D}]
We have
\begin{align*}
\dist \big( A^{(n)}(\omega) \phi(\omega), \phi(F^n \omega) \big) &\le 
\sum_{j=0}^{n-1} \dist \big( A^{(n-j)}  (F^{j}\omega)   \phi(F^{j}\omega), 
                             A^{(n-j-1)}(F^{j+1}\omega) \phi(F^{j+1}\omega) \big) \\
&=
\sum_{j=0}^{n-1} \dist \big( A(F^j \omega)\phi(F^j\omega) ,  \phi(F^{j+1}\omega)\big) \\
&\le n \displ(\phi).
\end{align*}
Dividing by $n$ and passing to the limit, we obtain
$\drift(F,A) \le \displ(\phi)$, as desired.
\end{proof}

\subsection{Preliminaries}

The proof of \cref{t.slow_D} requires the preliminaries below.

\subsubsection{Busemann spaces}\label{sss.Busemann}

Let $(H, \dist)$ be a separable metric space.
We say that $H$ is a \emph{geodesic space} 
if it is complete and every two points $p$, $q$ in $H$ can joined by a \emph{geodesic}, 
that is, a curve $\gamma \!: [0, 1] \to H$ such that $\gamma(0) = p$, $\gamma(1)=q$, 
and $\dist(\gamma(t), \gamma(s)) = |t-s| \dist(p,q)$ for all $s,t$ in $[0,1]$. 
If these curves are unique for arbitrarily prescribed $p,q$, 
then we say that $H$ is \emph{uniquely geodesic}.

The space $H$ has {\em nonpositive curvature in the 
sense of Busemann} (it is a {\em Busemann space}, for short) 
if it geodesic and the distance function along geodesics is 
convex. Equivalently, given any two pairs of points $p$, $q$ and 
$p'$, $q'$, their corresponding midpoints $m := \mid (p,q)$ 
and $m' := \mid (p',q')$ satisfy
\begin{equation}\label{e.media}
\dist (m,m') \leq \frac{\dist (p,p')}{2} + \frac{\dist (q,q')}{2}.
\end{equation}
The family of such spaces obviously includes all strictly convex 
Banach spaces. 
(General Banach spaces may also be included in this category when considering only segments of lines as geodesics.) 
It also includes complete simply connected Riemannian manifolds of nonpositive curvature
(such as those that will appear in the proof of \cref{t.matrix}).
For infinite-dimensional examples, see the remarks in \cref{ss.matrix}.

\subsubsection{Barycenter maps}\label{sss.bary}

Given a metric space $(H,\dist)$, we denote by $\mathcal{P}^1 (H)$ the space 
of probability measures on $H$ with finite first moment, that is, 
such that for some (equivalently, for every) $p_0 \in H$ one has 
$$
\int_H \dist (p_0,p)\, \mathrm{d}\mu (p) < \infty.
$$
We endow this space with the $1$-Wasserstein metric $W_1$ defined as 
\begin{equation}
W_1(\mu,\nu) := \inf_{P \in (\mu \mid \nu)} \int_{H \times H} 
\dist (p,q) \, \mathrm{d}P (p,q),
\end{equation}
where $(\mu \mid \nu)$ denotes the set of all 
{\em couplings} of $\mu$ and $\nu$, that is, all 
probability measures $P$ on $H \times H$ whose projection along the 
first (resp.\ second) coordinate coincides with $\mu$ (resp.\ $\nu$). 

\begin{example}\label{exa.wass}
If $\mu$, $\nu_1$, $\nu_2 \in \mathcal{P}^1 (H)$ then for every $\lambda \in [0,1]$,
$$
W_1 \big( (1-\lambda)\mu + \lambda\nu_1, (1-\lambda)\mu + \lambda\nu_2 \big) \le 
\lambda \cdot \sup\{\dist(p_1,p_2) \, \!: p_i \in \supp \nu_i\}.
$$
Indeed, this follows directly from the definitions by considering the coupling 
$(1-\lambda) i_* \mu + \lambda \, \nu_1 \times \nu_2${\,}, where
$i(p) := (p,p)$.
\end{example}

For much more on Wasserstein metrics, see e.g.\ \cite{Villani}.

Let $\delta_p$ denote Dirac measure on $p$.

\begin{theorem} \label{t.bary} 
If $H$ is a Busemann space, then there exists a map $\bary \!: \mathcal{P}^1 (H) \to H$ 
that satisfies $\bary(\delta_p) = p$ for each $p$, 
is equivariant with respect to the action of the isometries,
and is $1$-Lipschitz for the $1$-Wasserstein metric on $\mathcal{P}^1 (H)$.
\end{theorem}

In particular, in the situation of \cref{exa.wass}, we have
\begin{equation}\label{e.bar_lips_C}
\dist \big( \bary((1-\lambda)\mu + \lambda\nu_1), 
\bary((1-\lambda)\mu + \lambda\nu_2) \big) \le 
\lambda \cdot \sup\{\dist(p_1,p_2) \, \!:  p_i \in \supp \nu_i\}.
\end{equation}

We also denote
$$
\bary(p_1, \dots, p_n) := \bary \frac{1}{n}(\delta_{p_1} + \cdots + \delta_{p_n}),
$$
It follows from (\ref{e.bar_lips_C}) that  
\begin{equation}\label{e.bar_lips_D}
\dist \big( \bary(p_1, \dots, p_{n-1}, p_n), \bary(p_1, \dots, p_{n-1}, p_n') \big) 
\le \frac{1}{n} \dist(p_n,p_n').
\end{equation}
	
\medskip

Henceforth, we will not need to know what precisely is the map $\bary$ above, 
although its geometrical flavor should be intuitively transparent. For instance,
$\bary(p)$ coincides with $p$, while $\bary(p_1,p_2)$ is the midpoint between 
$p_1$ and $p_2$. The definition of $\bary(p_1,p_2,p_3)$ is, however, quite involved.

In its full generality, Theorem \ref{t.bary} above was proved by Navas in \cite{Navas} by 
elaborating on an idea introduced by Es-Sahib and Heinich in \cite{EH}. Nevertheless, 
for compactly supported measures $\mu$ on CAT(0)-spaces, 
a much more classical notion 
of barycenter due to Cartan (see \cite[Note~III, Part~IV]{Cartan}) is enough for our purposes. 
(See \cref{sss.alexandrov} for the definition of CAT(0) spaces; see also \cite{BK, Jost, ArLi}.)
The Cartan barycenter of $\mu$ as above is defined as the unique point that minimizes the function
\begin{equation}\label{e.sum_squares}
f_\mu(p) = \int_H \dist^2(q,p) \, \mathrm{d}\mu(q).
\end{equation}
The fact that Cartan's barycenter is $1$-Lipschitz 
with respect to the $1$-Wasserstein metric is proved in \cite{Sturm}.

\subsection[Proof of \cref{t.slow_D}]{Existence of sections of nearly minimal displacement: 
proof of \cref{t.slow_D}}\label{ss.proof_slow_D}

\begin{proof}[First proof of \cref{t.slow_D}]
Given a cocycle $A$ of isometries of a Busemann space $H$ over 
$F \!: \Omega \to \Omega$, we fix any $p_0 \in H$. Let 
\begin{equation}\label{e.magic2}
\phi_N (\omega) := 
\bary \left( p_0, A(\omega)^{-1} p_0 , [A^{(2)}(\omega)]^{-1} p_0 , \dots, 
[A^{(N-1)}(\omega)]^{-1} p_0 \right),
\end{equation}
where $\bary$ stands for the barycenter introduced in \cref{t.bary}. 
Then, by equivariance of the barycenter, we have:   
\begin{small}\begin{align*}
A(\omega)\phi_N (\omega) &= 
\bary \left( A(\omega)  p_0, p_0, A(F\omega)^{-1} 
p_0, \dots, [A^{(N-2)}(F\omega)]^{-1} p_0
\phantom{, [A^{(N-1)}(F\omega)]^{-1} p_0} \right),
\\
\phi_N (F \omega) &= 
\bary \left( \phantom{A(\omega)  p_0,} p_0, A(F\omega)^{-1} p_0, 
\dots, [A^{(N-2)}(F\omega)]^{-1} p_0, [A^{(N-1)}(F\omega)]^{-1} p_0  \right).
\end{align*}\end{small}

Using \eqref{e.bar_lips_D}, we obtain 
\begin{align*}
\dist \big( A(\omega) \phi_N (\omega) , \phi_N (F \omega) \big) 
&\le \frac{1}{N} \dist \big(A(\omega)  p_0, [A^{(N-1)}(F\omega)]^{-1} p_0 \big) \\
&=   \frac{1}{N} \dist \big(A^{(N)}(\omega)  p_0, p_0 \big). 
\end{align*}
In particular,
$$
\displ(\phi_N) \le \frac{1}{N} \sup_{\omega \in \Omega} 
\dist \big(A^{(N)}(\omega) p_0, p_0 \big). 
$$
The theorem then follows by taking $\phi=\phi_N$ with 
sufficiently large $N$ (depending on $\varepsilon$).
\end{proof}

\begin{remark}\label{r.family}
We can extend the family $\{\phi_N\}_{N \in \Z_+}$ to a family $\{\phi_t\}_{t\in \R}$ by letting
$$
\phi_t(\omega) := 
\bary \big( (1-t+N) \delta_{\phi_N(\omega)} +  (t-N) \delta_{\phi_{N+1}(\omega)}\big),
\quad \text{where }
N = \lfloor t \rfloor.
$$
Then $\lim_{t \to \infty} \displ (\phi_t) = \drift(F,A)$;
moreover, $\phi_t$ depends continuously on $t$ and also on $A$.
(Such parameterized sections play an important role for the 
particular case considered in \cite{ABD2}.)~\closeremark 
\end{remark}

\medskip

As the referee informed us, \cref{t.slow_D} was already known in the case that $\Omega$ is a point; 
see \cite[Lemma~6.6]{BGS}. In fact, it is possible to adapt the argument therein to give another 
proof of \cref{t.slow_D}:

\begin{proof}[Second proof of \cref{t.slow_D}]
Let $\phi_0 \colon \Omega \to H$ be any section (e.g., constant equal to the basepoint $p_0$).
Then 
$$
\drift(F,A) = \lim_{n\to \infty} \frac{1}{n} \displ_{F^n, A^{(n)}} (\phi_0).
$$

Take $n$ of the form $2^k$ such that $\frac{1}{n} \displ_{F^n, A^{(n)}} (\phi_0) < \drift(F,A) + \epsilon$.
If $k=0$, then we are done, so assume that $k \ge 1$.

Define a new section by
$$
\phi_1 (\omega) := \mid \Big[A^{(n/2)}(F^{-n/2} \omega) \phi_0(F^{-n/2} \omega) , \ \phi_0(\omega)\Big] \, ,
$$
where $\mid$ stands for the midpoint of a segment.
Then 
\begin{align*}
A^{(n/2)}(\omega) \phi_1 (\omega)
&= \mid \Big[ A^{(n)}(F^{-n/2} \omega) \phi_0(F^{-n/2} \omega) , \ A^{(n/2)}(\omega) \phi_0(\omega) \Big] \, , \\
\phi_1 (F^{n/2} \omega) 
&= \mid \Big[A^{(n/2)}(\omega) \phi_0(\omega) , \ \phi_0(F^{n/2} \omega) \Big] \, .
\end{align*}
Thus, by the Busemann property,
$$
d \Big( A^{(n/2)}(\omega) \phi_1 ( \omega), \phi_1 (F^{n/2} \omega) \Big) 
\le d \Big( A^{(n)}(F^{-n/2} \omega) \phi_0(F^{-n/2} \omega) , \  \phi_0(F^{n/2} \omega) \Big) \, ; 
$$
in particular,
$$
\displ_{F^{n/2}, A^{(n/2)}} (\phi_1) \le \frac{1}{2} \displ_{F^n, A^{(n)}} (\phi_0) \, .
$$

Repeating this construction, we recursively find sections $\phi_2$, \dots, $\phi_k$ such that
$$
\displ_{F^{n/2^j}, A^{(n/2^j)}} (\phi_j) \le \frac{1}{2} \displ_{F^{n/2^{j-1}}, A^{(n/2^{j-1})}} (\phi_{j-1}) 
\text{ for } 1 \le j \le k.
$$
Therefore,
$$
\displ_{F,A} (\phi_k) \le \frac{1}{2^k} \displ_{F^{n}, A^{(n)}} (\phi_0) < \drift(F,A) + \epsilon,
$$
that is, $\phi:=\phi_k$ has the required properties.
\end{proof}

The proof above does not require general barycenters.
On the other hand, while it seems feasible to adapt this proof to the continuous case,
it is unclear whether or not it can be adapted to more complicate group actions.

\subsection{More preliminaries}

The proof of \cref{t.closing_D} requires additional preliminaries.

\subsubsection{Topologies}\label{sss.topologies}

On the space $\Isom(H)$, we consider the first-countable topology 
for which the convergence of sequences is uniform convergence on bounded subsets.
This is called the \emph{bounded-open topology}.\footnote{If $H$ is proper then this coincides
with the compact-open topology, which is the usual topology on $\Isom(H)$; see e.g.~\cite{Helgason}.}

We also endow the set $C(\Omega, \Isom(H))$ of continuous functions from 
$\Omega$ into $\Isom(H)$ with the compact-open topology. 
A sequence $(A_n)$ converges to $A$ in this topology if and only if for every 
bounded subset $B$ of $H$, the sequence $(A_n(\omega)(p))$ converges to 
$A(\omega)(p)$ uniformly with respect to $(\omega,p) \in \Omega \times B$.

\subsubsection{Translation length}\label{sss.TL}

Recall that the \emph{displacement function} 
of $J\in \Isom(H)$ is the function 
\begin{equation}
p \in H \mapsto \dist(J(p),p).
\end{equation}
Since $H$ is a Busemann space, this function is convex. 
The infimum of the displacement function is called the 
\emph{translation length} of $J$.

\begin{remark}\label{r.drift_as_TL}
Let $A \!: \Omega \to \Isom(H)$ be a cocycle of isometries over a \emph{homeomorphism}
$F \!: \Omega \to \Omega$.
Let us explain how its drift can be seen as a translation length
of a certain isometry.

Let $C(\Omega, H)$ be the set of sections, endowed with the distance
$d(\phi_1,\phi_2) = \sup_{\omega} d(\phi_1(\omega), \phi_2(\omega))$.
(This is a geodesic space, but not a uniquely geodesic one.)
Let $\Gamma = \Gamma_{F,A} \!: C(\Omega, H) \to C(\Omega, H)$ be the \emph{graph transform} defined by
$(\Gamma \phi) (\omega) := A(\omega) \phi(F^{-1} \omega)$.
Then $\Gamma$ is an isometry of $C(\Omega, H)$, and $\displ(\phi) = d(\Gamma \phi, \phi)$.
Therefore \cref{t.slow_D} states that \emph{the drift of a cocycle of isometries
equals the translation length of the associated graph transform.}~\closeremark
\end{remark}

\subsubsection{Symmetric geodesic spaces}\label{sss.symm}

We say that a uniquely geodesic space $H$ is \emph{geodesically complete} if 
the maximal interval of definition of all geodesics is $\R$. 
For such a space, the \emph{symmetry} at a point $p_0 \in H$ 
is the map $\sigma_{p_0} \!: H \to H$ 
that sends $p$ to to the unique point $p'$ such that
$p_0$ is the midpoint between $p$ and $p'$.
So $\sigma_{p_0}$ is an involution.
We say that $H$ is a \emph{symmetric geodesic space} 
if $\sigma_{p_0}$ is a isometry for every $p_0 \in H$,
and the map $(p_0,p) \mapsto \sigma_{p_0}(p)$ is continuous.

\subsubsection{Transvections and a displacement estimate}\label{sss.transvections}

Assume that $H$ is a symmetric geodesic space. Following \'{E}.~Cartan, 
we call a \emph{transvection} an isometry of the form 
$J = \sigma_{p_2} \circ \sigma_{p_1}$. If $\gamma:\R \to H$ is a isometric 
(unit-speed) parametrization of the geodesic passing through $p_1$ and $p_2$, 
say with $\gamma^{-1}(p_1)<\gamma^{-1}(p_2)$, then $J(\gamma(t)) = \gamma(t + b)$ 
holds for all $t\in\R$, where $b = 2\dist(p_1,p_2)$. We say that $J$ translates 
the geodesic $\gamma$ by length $b$.

We remark that if, in addition, $H$ is a Busemann space, then
$\dist (J(q),q) \ge b$ for all $q \in H$; see \cite{Papa}. 
So the translation length of $J$ is precisely $b$.

\begin{lemma}\label{l.displ_estimate_proper}
Assume that $H$ is a symmetric geodesic space.
Also assume that $H$ is proper. Let $J$ be a 
transvection that translates a geodesic 
$\gamma$ by length $b$. Then
\begin{equation}\label{e.displ_estimate}
\dist (J(q), q ) \le f\left( b, \dist(q, \gamma) \right) \text{ for every $q\in H$,}
\end{equation}
where $f \!: \R_+ \times \R_+ \to \R_+$ is a function that depends only on the space $H$,
and is monotonically increasing with respect to each variable.
\end{lemma}

\begin{proof} 
Define a function $\tilde f \!: H \times \R_+ \times \R_+ \to \R_+$ by 
$$
\tilde f (p_0, b, \ell) := \sup \big\{ \dist(J(q),q) \, \!:  
J = \sigma_{p_1} \circ \sigma_{p_0}, \ 
\dist(p_1,p_0) \le b/2, \ 
\dist(q,p_0) \le \ell
\big\} \, .
$$
The supremum is finite by properness of $H$ and continuity.
Since the group of isometries acts transitively on $H$,
the value $\tilde f (p_0, b, \ell)$ actually does not depend on $p_0$;
call it $f(b, \ell)$.
Then \eqref{e.displ_estimate} holds.
\end{proof}

\begin{remark}\label{r.continuity}
It follows that under the assumptions of \cref{l.displ_estimate_proper}, 
the map $p_0 \in H \mapsto \sigma_{p_0} \in \Isom(H)$
is continuous (where $\Isom(H)$ is endowed with 
the bounded-open topology, as explained in \cref{sss.topologies}.)~\closeremark 
\end{remark}

\subsubsection[Curvature bounds]{Curvature bounds in the sense of 
Alexandrov, and a displacement estimate}\label{sss.alexandrov}

In the case that $H$ is infinite dimensional, 
the proof of \cref{l.displ_estimate_proper} given above obviously does not work.
Nevertheless, the lemma holds if properness is replaced by some curvature
hypotheses, as we next explain.
Readers who are not interested in infinite dimensional applications can skip this paragraph.

\medskip

Given $\kappa \le 0$, the \emph{model space} $(M_\kappa,d_\kappa)$ is the 
two-dimensional space of constant curvature $\kappa$.\footnote{We will not consider $\kappa>0$ in order to avoid unnecessary complications.}

Let $H$ be a uniquely geodesic space.
A \emph{triangle} $\triang(p_1,p_2,p_3)$ in $H$ consists on three points 
$p_1$, $p_2$, $p_3$ and three geodesic segments joining them.
Suppose that $\triang(\tilde p_1, \tilde p_2, \tilde p_3)$ is a triangle in the model space $M_\kappa$
such that $\dist(p_i,p_j) = \dist_\kappa(\tilde p_i, \tilde p_j)$ for all $i,j$ 
in $\{1,2,3\}$. Then we say that  $\triang(\tilde p_1, \tilde p_2, \tilde p_3)$  is a
\emph{SSS-comparison triangle}\footnote{SSS stands for \emph{side-side-side}.} 
for $\triang(p_1,p_2,p_3)$.

We say that $H$ has \emph{curvature $\le \kappa$ (resp.\ $\ge \kappa$) in the sense of Alexandrov} 
if for every triangle 
$\triang(p_1,p_2,p_3)$ in $H$ and every SSS-comparison triangle 
$\triang(\tilde p_1, \tilde p_2, \tilde p_3)$ in the model space $M_\kappa$, 
the following inequality holds for all $t \in [0,1]$:
$$
\dist \big( p_3, (1-t) p_1 + t p_2 \big) \le \text{  (resp.\ $\ge$) }
\dist_\kappa \big( \tilde p_3, (1-t) \tilde p_1 + t \tilde p_2 \big),
$$
where $t \mapsto (1-t) p_1 + t p_2$ is a short-hand for the geodesic segment joining $p_1$ and $p_2$.

\begin{remark}\label{r.local_global}
Actually the usual definition requires only local comparisons;
however, (in the cases that we consider here) 
this turns out to be equivalent to our (global) definition; 
see \cite[\S~4.6.2]{BBI} and references therein.~\closeremark
\end{remark}

Spaces of curvature $\leq 0$ are also called 
{\em CAT(0)-spaces}. It is a standard fact that every complete CAT(0) space 
is a Busemann space (see \cite[p.~176]{BH} or \cite[Corollary 2.5]{Sturm}).

We will say that an uniquely geodesic space $H$ has 
\emph{bounded nonpositive curvature in the sense of Alexandrov} 
if it has curvature $\le 0$ and $\ge \kappa$ for some $\kappa \le 0$. 

\medskip

Now we have the following version of \cref{l.displ_estimate_proper}: 

\begin{lemma}\label{l.displ_estimate_nonproper}
Assume that $H$ is a symmetric space of bounded nonpositive curvature in the sense of Alexandrov. 
Let $J$ be a transvection that translates a geodesic $\gamma$ by length $b$. Then
\begin{equation}\label{e.displ_estimate_again}
\dist (J(q), q ) \le f\left( b, \dist(q, \gamma) \right) \text{ for every $q\in H$,}
\end{equation}
where $f \!: \R_+ \times \R_+ \to \R_+$ is a function that depends only on the space $H$,
and is monotonically increasing with respect to each variable.
\end{lemma}

We leave the proof of this lemma to Appendix~\ref{s.alexandrov}.

\begin{remark}\label{r.continuity_again}
Similarly to \cref{r.continuity}, we conclude that under the 
assumptions of \cref{l.displ_estimate_nonproper}, the map 
$p_0 \in H \mapsto \sigma_{p_0} \in \Isom(H)$ is continuous.~\closeremark 
\end{remark}

\subsubsection{Macroscopic uniform homogeneity}\label{sss.homog_macro}

Every symmetric space is homogeneous (in the 
sense that the group of isometries acts transitively).
We will need however a stronger property, given by the following lemma:

\begin{lemma}[Macroscopic uniform homogeneity]\label{l.homog_macro}
Assume that $H$ is 
\begin{itemize}
\item either a proper Busemann space;
\item or a space of bounded nonpositive curvature  in the sense of Alexandrov.
\end{itemize}
Also assume that $H$ is symmetric. Then there exists a continuous map 
$J  \!:  H \times H \to \Isom (H)$ with the following properties:
\begin{enumerate}
\item 
$J(p,q)p=q$ for all $p$, $q$ in $H$.
\item\label{i.convergence_m} 
$J(p,q)$ converges to the identity as the distance between $p$ and $q$ converges to zero.
\end{enumerate}
\end{lemma}

More explicitly, assertion~(\ref{i.convergence_m}) means 
that for every $\epsilon > 0$ and each bounded subset $B \subset H$, there exists 
$\delta > 0$ such that $\dist(J(p,q)r,r) < \epsilon$ holds for all $r \in B$ 
whenever $\dist (p,q) < \delta$. 
(Notice that $p$ and $q$ are not restricted to a bounded set.)

\begin{proof}
Fix some $p_0 \in H$, and consider the transvection (see \cref{f.homog}):
\begin{equation}\label{e.homog}
J(p,q) := \sigma_m \circ \sigma_{p_0} \, ,
\quad \text{where $m$ is the midpoint between $\sigma_{p_0}(p)$ and $q$.}
\end{equation}

\psfrag{p}[r][r]{$p$}
\psfrag{q}[r][r]{$q$}
\psfrag{0}[l][l]{$p_0$}
\psfrag{b}[l][l]{$m$}
\psfrag{s}[l][l]{$\sigma_{p_0}(p)$}
\psfrag{1}[r][r]{$\sigma_{p_0}$}
\psfrag{2}[l][l]{$\sigma_{m}$}
\begin{figure}
\begin{center}
\includegraphics[width=.7\textwidth]{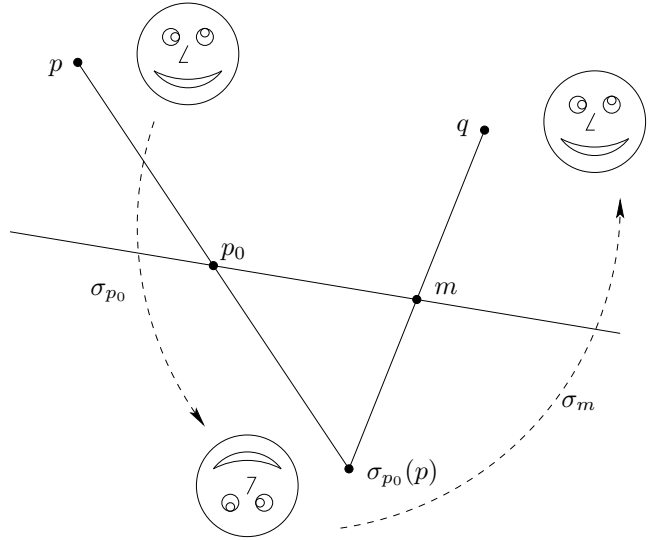}
\end{center}
\caption{$m$ is the midpoint between $q$ and $\sigma_{p_0}(p)$. 
The isometry $J(p,q) := \sigma_m \circ \sigma_{p_0}$ sends 
$p$ to $q$ and translates the geodesic joining $p_0$ and 
$m$ by length $2\dist(p_0,m) \le \dist(p,q)$.}
\label{f.homog}
\end{figure}

Applying the Busemann inequality \eqref{e.media}
to the points $\sigma_{p_0}(p)$, $p$ and $q$, we obtain
$\dist(p_0, m) \le \tfrac{1}{2} \dist(p,q)$. Therefore, the length 
by which the transvection $J(p,q)$ translates $\gamma$ is at most 
$\dist(p,q)$. So assertion~(\ref{i.convergence_m}) follows directly from 
\cref{l.displ_estimate_proper} or \cref{l.displ_estimate_nonproper},
according to the case.
\end{proof}

\begin{remark}
In the case where $H$ is the hyperbolic plane, 
\cref{l.homog_macro} follows from Lemma~5 from \cite{ABD1}. 
Although the construction presented therein 
is specific to the hyperbolic plane,
it actually produces the same isometries as our formula \eqref{e.homog} 
in this particular case. 
\closeremark 
\end{remark}

\begin{remark}
Despite the fact that the perturbative argument that appears in \cite{BN_elementary}
is elementary and does not allude to any geometry,
it is actually the construction above specialized to $H = \GL(d,\R)/\mathrm{O}(d)$.~\closeremark
\end{remark}

\subsection[Proof of \cref{t.closing_D}]{Creating invariant sections: 
proof of \cref{t.closing_D}}\label{ss.proof_closing_D}

\begin{proof}[Proof of \cref{t.closing_D}] 
By \cref{t.slow_D}, there exists a sequence of sections $\phi_N$
such that $\lim_{N \to \infty} \displ(\phi_N) = 0$. Define
\begin{equation}\label{e.closing_D}
\tilde A_N(\omega) = 
J \big( A(\omega)\phi_N(\omega),\phi_N(F\omega) \big) \circ A(\omega),
\end{equation}
where $J$ is given by \cref{l.homog_macro}.
Then:
\begin{itemize}
\item $\tilde A_N(\omega)  \phi_N(\omega) = \phi_N(F \omega)$, 
that is, $\phi_N$ is $\tilde A_N$-invariant.
\item for each bounded subset $B$ of $H$,
the sequence $\tilde A_N (\omega)  p$ converges to 
$A(\omega) p$ uniformly with respect to $(\omega,p) \in \Omega \times B$.
\end{itemize}
This shows the theorem except for the claim concerning the cohomologous cocycle. 
To prove this last issue, take any point $p_0 \in H$ and consider the cocycle  
$B_N(\omega) = U(F \omega)^{-1} \circ A_N(\omega) \circ U(\omega)$,
where $U(\omega) = J(p_0, \phi_N(\omega))$. Then $B_N$ is cohomologous 
to $A_N$ and takes values in the stabilizer of $p_0$, as desired.
\end{proof}

\begin{remark}\label{r.non-perturbative}
A ``non-perturbative'' version of \cref{t.closing_D} goes as follows:
\emph{Under the hypotheses of \cref{t.closing_D}, it follows that $A$ is cohomologous to cocycles arbitrarily close to cocycles taking values in stabilizers of points.} 

Indeed, by \cref{t.slow_D}, there exists a sequence of sections $\phi_N$ such that $\lim_{N \to \infty} \displ(\phi_N) = 0$. Take any point 
$p_0 \in H$ and consider the cocycle $B_N(\omega) = U(F \omega)^{-1} \circ A (\omega) \circ U(\omega)$, where 
$U(\omega) = J(p_0, \phi_N(\omega))$
and $J$ is provided by \cref{l.homog_macro}.
If $N$ is large then 
\begin{multline*}
d \big( B_N (\omega) p_0, p_0 \big) 
= d \big( A(\omega) \circ  U(\omega) p_0, U (F(\omega)) p_0 \big) \\
= d \big( A(\omega) \varphi_N (\omega), \varphi_N (\omega) \big) 
\le \displ(\varphi_N)
\end{multline*}
is small.
Thus $B_N$ is a cocycle cohomologous to $A$ close to a cocycle taking values in the stabilizer of~$p_0$.~\closeremark
\end{remark}

\begin{remark}\label{r.access}
Let us see how to obtain certain ``accessibility'' properties, which play 
an important role for the particular case treated in \cite{ABD2}.
First, by \cref{r.family} we can find a continuous family $\{\phi_t\}_{t\in \R}$ of sections
such that $\displ_{F,A}(\phi_t) \to 0$ as $t \to + \infty$.
Repeating the construction of the proof of \cref{t.closing_D}, 
we conclude the following: \emph{For every cocycle $A$ with uniform sublinear drift, 
there exists a continuous family of cocycles $\{A_t\}_{t\in [0,\infty]}$, satisfying 
$A_{\infty} = A$ and such that for each $t<\infty$, $A_t$ has a continuous invariant 
section $\phi_t$ (that also depends continuously on $t$). Moreover, such correspondence  
is continuous: given a continuous family $A(s)$ of cocycles ($s$ in an arbitrary 
topological space), the resulting $A_t(s)$ and $\phi_t(s)$ are jointly continuous.}~\closeremark
\end{remark}

\begin{remark}\label{r.fibered_D}
Replace $\Omega \times H$ by a fiber bundle $\Sigma$ with base space $\Omega$, fiber $H$,
and structural group $\Isom(H)$.
Then the mappings $\Sigma \to \Sigma$ that preserve the bundle structure and 
project over $F$ play the role of the cocycles of isometries.
Fibered versions of \cref{t.slow_D,t.closing_D} actually hold.
The proofs are basically the same, replacing the basepoint $p_0$
that appears (explicitly or implicitly) in the fundamental formulas
\eqref{e.drift_D}, \eqref{e.magic2}, \eqref{e.homog}, and \eqref{e.closing_D}
by any continuous section $\Omega \to H$.~\closeremark
\end{remark}

\subsection{Application to matrix cocycles}\label{ss.matrix}

In the proof below, we use some geometrical facts
that can be found in Chapter~II.10 (especially p.~328--329) of \cite{BH};
see also Chapter~XII of \cite{Lang}.

\begin{proof}[Proof of \cref{t.matrix}]
Let $G$ be an algebraic subgroup of $\GL(d,\R)$ that is closed under matrix transposition,
and let $K = G \cap \mathrm{O}(n)$.
Consider the action of $G$ on the space $H:=G/K$ of left cosets.
Then we can metrize $H$ so the action becomes isometric, 
and moreover $H$ becomes a symmetric Busemann space.
Actually, for each $g \in G$, the distance between the cosets $gK$ and $K$ is 
$\left( \sum (\log \sigma_i)^2 \right)^{1/2}$, where $\sigma_1$, \dots, $\sigma_d$
are the singular values of the matrix $g$.

Now let $A \!: \Omega \to G$ be a cocycle, and let $[A] \!: \Omega \to \Isom(H)$
be the induced cocycle of isometries.
Assume that $A$ has uniform subexponential growth.
It follows the distance formula above that $[A]$ has uniform subexponential drift. 

By \cref{t.closing_D}, there is a perturbation of $[A]$ that has an invariant section 
$\phi \!: \Omega \to H$.
Actually, this perturbation is obtained by composition with transvections
(recall \eqref{e.closing_D} and \eqref{e.homog}),
which are induced by elements of $G$ (see \cite[Lemma~1, p.~235]{KN}).
So the perturbed cocycle of isometries is induced by a perturbation $\tilde A$ of the original $G$-cocycle.

Choose (e.g., using \cref{l.homog_macro}) a continuous map $U \!: \Omega \to G$ such that 
for each $\omega \in \Omega$, the coset containing $U(\omega)$ is precisely $\phi(\omega)$.
Then $U$ is a conjugacy between $\tilde A$ and a $K$-valued cocycle,
as desired.
\end{proof}

\begin{remark}
It is actually possible to state \cref{t.matrix} in a Lie group setting,
and prove it using e.g.~\cite[Thrm.~8.6(2), p.~256]{KN}.
We preferred, however, to keep the statements simpler,
relying only on more elementary results as those from 
\cite{BH} or \cite{Lang}.~\closeremark
\end{remark}

\begin{remark}\label{r.infty_dim}
Let $\mathcal{H}$ be an infinite-dimensional separable real Hilbert space.
Let $\GL^2(\infty, \R)$ be the group of all invertible operators on 
$\mathcal{H}$ that may be written in the form $\Id + L$, where 
$L$ is a Hilbert--Schmidt operator.
Let $\mathrm{O}^2(\infty)$ be the orthogonal subgroup of $\GL^2(\infty, \R)$. 
Then $H := \GL^2(\infty, \R) / \mathrm{O}^2(\infty)$ 
can be given a structure of symmetric Cartan--Hadamard manifold
on which $\GL^2(\infty,\R)$ acts by isometries;
see \cite{Larotonda}. 
In particular, $H$ has bounded nonpositive curvature in the sense of Alexandrov
(see \cite[Section 7]{KarM} for more on this; see also \cref{r.symm_bound,r.sectional}).
Hence \cref{t.slow_D,t.closing_D} apply to this space. 
In particular, \cref{t.matrix} extends to cocycles of this kind of operators.~\closeremark
\end{remark}

\begin{remark}
Consider now the space $H = \GL(\infty,\R)/\mathrm{O}(\infty)$, where $\GL(\infty,\R)$ is the group of all 
bounded invertible operators on $\mathcal{H}$, and $\mathrm{O}(\infty)$ is the orthogonal subgroup. 
It is possible (see \cite{CPR,LL}) to metrize $H$ so it becomes a Busemann space,
besides being a symmetric space and a Banach manifold;
however the resulting space is not CAT(0).
Therefore \cref{t.slow_D} applies to the space $H$.
However, we do not know whether \cref{t.closing_D} applies to this space,
or whether \cref{t.matrix} applies to $\GL(\infty,\R)$-cocycles.~\closeremark 
\end{remark}	

\section{Cocycles over other group actions} \label{s.proofs_G}

We now consider other (still discrete) group actions.
Before going into the proofs of our results, let us 
make an observation about the construction of almost invariant sections.

\medskip

Let $\Gamma$ be a (non necessarily abelian) group acting by homeomorphisms of a 
compact Hausdorff metric space $\Omega$. 
Let $A$ be a cocycle over this group action 
with values in the group of isometries of a Busemann space $H$.

Suppose that $C_N$ is a sequence of finite subsets of $\Gamma$.
Fix $p_0 \in \Omega$ and define a sequence of sections $\varphi_N \!: \Omega \to H$ by 
\begin{equation}\label{e.magic_group}
\varphi_N (\omega) := 
\bary \Big( \big(A^{(h)} (\omega)\big)^{-1} p_0 \!: h \in C_N \Big),
\end{equation}
where $\bary$ is provided by \cref{t.bary} 
(compare with \eqref{e.magic2}).
Now fix any $g \in \Gamma$.
By equivariance of the barycenter,
\begin{equation}\label{e.image}
A^{(g)} (\omega) \varphi_N (\omega) := 
\bary \Big( \big(A^{(h)} (g\omega)\big)^{-1} p_0 \!: h \in C_N \cdot g^{-1} \Big),
\end{equation}
By property \eqref{e.bar_lips_D} of the barycenter map, we have
\begin{align}
d \big( A^{(g)} (\omega) \varphi_N (\omega), \varphi_N (g \omega) \big) 
&\le \frac{n}{|C_N|}
\max_j d \Big( \big(A^{(h_j)} (g\omega)\big)^{-1} p_0 , \big(A^{(h_j')} (g\omega)\big)^{-1} p_0  \Big) \notag \\
&\le \frac{n}{|C_N|}
\max_j d \Big( A^{(h_j' h_j^{-1})} (h_j g\omega) p_0 ,  p_0  \Big),
\label{e.group_estimate}
\end{align}
where $n$ is the cardinality of the union of 
$(C_N \cdot g^{-1}) \setminus C_N$ and $C_N \setminus (C_N \cdot g^{-1})$, 
which are enumerated as 
$\{ h_1, \dots, h_n\}$ and $\{ h_1', \dots, h_n'\}$, respectively.

\subsection{Proof of \cref{t.slow_G}}

\begin{proof}[Proof of Theorem~\ref{t.slow_G}]
We first consider the case where $\Gamma = \mathbb{Z}^d$. Let  
$A \! : \Omega \times \mathbb{Z}^d \to \Isom (H)$ \hspace{0.01cm} 
be a cocycle of isometries of a Busemann space $H$. 
Assume that $A$ has uniform sublinear growth along cyclic subgroups.
We need to exhibit a sequence of continuous maps 
$\varphi_N \!: \Omega \to H$ such that
for all $i = 1,\ldots,d$, 
\begin{equation}\label{e.al-inv}
\lim_{N\to \infty} d \big( A^{(e_i)} (\omega) \varphi_N (\omega), \varphi_N (e_i \omega) \big) 
= 0 \quad \text{uniformly on $\omega \in \Omega$},
\end{equation}
where $e_i := (0,\ldots,0,1,0,\ldots,0)$ is the $i^\text{th}$ canonical generator of $\mathbb{Z}^d$. 

To do this, consider the sequence of ``cubes''
$$
C_N := \big\{ (m_1, \dots, m_d) \in \Z^d \!: 0 \leq m_j < N \big\},
$$
and define $\varphi_N$ by \eqref{e.magic_group}.
The sets $(C_N - e_i) \setminus C_N = \{h_j\}$ and 
$C_N \setminus (C_N - e_i)$ have cardinality $n = N^{d-1} = |C_N|/N$.
Moreover, they can be enumerated respectively as  
$\{ h_1, \dots, h_n\}$ and $\{ h_1', \dots, h_n'\}$
in a way such that $h_j' - h_j = Ne_i$.
Then \eqref{e.group_estimate} gives 
$$
d \big( A^{(e_i)} (\omega) \varphi_N (\omega), \varphi_N (e_i + \omega) \big) 
\le \frac{2}{N}
\max_j d \Big( A^{(N e_i)} ((h_j + e_i ) \omega) p_0 ,  p_0  \Big),
$$
Since $A$ has sublinear drift along the cyclic subgroup generated by $e_i$,
\eqref{e.al-inv} follows.
This proves the theorem in the case $\Gamma = \Z^d$.

\medskip

Now consider the general case where $\Gamma$ is finitely generated and abelian.
Let $\Gamma = \Gamma_0 \oplus \mathbb{Z}^d$ 
be the torsion decomposition, where $\Gamma_0$ is the torsion subgroup.
Consider the sequence of sets
$$
C_N := \Gamma_0 \oplus \big\{ (m_1, \dots, m_d) \in \Z^d \!: 0 \leq m_j < N \big\},
$$
and define $\varphi_N$ by \eqref{e.magic_group}.
If $g \in \Gamma_0$ then $C_N - g = C_N$, and thus
\eqref{e.image} gives 
$A^{(g)} (\omega) \varphi_N (\omega) =  \varphi_N (g \omega)$.
On the other hand, if $g=e_i$ then
we can estimate as before 
$d \big( A^{(e_i)} (\omega) \varphi_N (\omega), \varphi_N (e_i \omega) \big) = o(N)$.
So $\varphi_N$ is a sequence of almost-invariant sections,
as we wanted.
\end{proof}


\subsection{Generalization to virtually nilpotent group actions}

We close this section with a further generalization of \cref{t.slow_G} 
for cocycles over virtually nilpotent group actions. 

Let $\Gamma$ be a finitely generated group acting on a compact space $\Omega$,
and let $A$ be a cocycle of isometries of a space $(H,d)$ over this action.
We say that $A$ has \emph{uniform sublinear drift}
if for each fixed $p_0 \in H$, 
$$
\sup_{\omega \in \Omega} d (A^{(g)} (\omega) p_0, p_0) = o(\ell(g)),
$$
where $\ell$ denotes word length with respect to some 
finite system of generators. 

\begin{theorem} \label{t.slow_nilp} 
Let $\Gamma$ be a finitely generated virtually nilpotent group acting by homeomorphisms of a 
compact Hausdorff metric space $\Omega$. 
Let $A$ be a cocycle over this group action 
with values in the group of isometries of a Busemann space $H$. 
If $A$ has sublinear drift 
then $A$ admits almost-invariant sections.
\end{theorem}

\begin{proof} 
We follow an argument of \cite{CTV}. Since $\Gamma$ is virtually nilpotent, it has polynomial growth 
(with respect to any finite system of generators). 
Denoting by $B(n)$ the ball of radius $n$ in $\Gamma$, 
we claim that there exist $D > 0$ and an increasing sequence of integers $k_N$ such that for all $N$:
\begin{equation}\label{e.nilp-drift}
\frac{ \big| B(k_N + 1) \setminus B(k_N) \big|}{\big| B(k_N) \big|} \leq \frac{D}{k_N}.
\end{equation}
Otherwise, for each $D > 0$ there would exist positive constants $C$, $C'$, $C''$ 
such that
$$
|B(k)| \geq
C \prod_{j=1}^{k-1} \left( 1 + \frac{D}{j} \right) 
\geq C' \exp \left( \sum_{j=1}^{k-1} \frac{D}{j} \right) 
\geq C'' k^D,
$$
thus contradicting polynomial growth.

Now fix $p_0 \in \Omega$, let $C_N := B(k_N)$,
and define $\varphi_N \!: \Omega \to H$ by \eqref{e.magic_group}.
Let $g$ be a generator of $\Gamma$.
Then, by \eqref{e.nilp-drift},
$$
n:=|(C_N \cdot g^{-1}) \setminus C_N| \le | B(k_N + 1) \setminus B(k_N) | \leq \frac{D}{k_N} |C_N|.
$$
Therefore, \eqref{e.group_estimate} gives
$$
\sup_{\omega \in \Omega} d \big( A^{(g)} (\omega) \varphi_N (\omega), \varphi_N (g \omega) \big) 
\le \frac{D}{k_N}
\sup_{\omega \in \Omega}  \max_{h \in B(2k_N + 1)}  d \big( A^{(h)} (\omega) p_0 ,  p_0  \big),
$$
which converges to $0$ as $N \to \infty$.
We conclude that $\varphi_N$ is a sequence of almost-invariant sections. 
\end{proof}

Together with \cref{t.slow_nilp}, the next general Proposition shows that Theorem E extends to 
virtually nilpotent groups. 

\vspace{0.1cm}

\begin{proposition} Let $A$ be a cocycle of isometries of an space $H$ over a group action 
by homeomorhisms on a space $\Omega$. If $\Gamma$ is virtually nilpotent, then $A$ 
has zero drift along cyclic subgroups if and only if it has uniform sublinear growth.
\end{proposition}

\begin{proof} 
Without loss of generality, we may assume that $\Gamma$ is torion-free and nilpotent. 
As it is well-known, such a group is boundedly generated in a strong form: there exists a 
generating system $\mathcal{G} = \{  h_1,\ldots, h_k \}$ and a constant $C$ such that every 
element $h \in \Gamma$ writes as $h = h_{i_1}^{n_1} \cdots h_{i_m}^{n_m}$, where each 
$h_{i_j}$ belongs to $\mathcal{G}$, $m \leq C$ and $|i_j| \leq C \ell (h)$. (In the torsion-free case, 
this follows, for instance, from \cite[Appendix B]{Breuillard-Green}.) 
Using this fact, the direct implication follows easily. The converse is 
straightforward and we leave it to the reader. 
\end{proof}

\begin{remark} For the case where $H$ is the real line and the cocycle is by translations, 
this yields an alternative (and simpler) proof of \cite[Th\'eor\`eme 2]{MOP2}.~\closeremark
\end{remark}

\begin{remark}
We do not know whether \cref{t.closing_D} may also be extended to (finitely generated) 
abelian or virtually nilpotent group actions.
The difficulty in adapting the proof is that the 
group relations must be preserved.
Of course, if we consider $\Gamma$ as a quotient of the free group
$\mathbb{F}_k$, where $k$ is the number of generators, 
and the action is lifted to  $\mathbb{F}_k$,
then the cocycle can be perturbed (as a cocycle above the $\mathbb{F}_k$-action) 
so that it has a continuous invariant section.
\closeremark
\end{remark}
%
%
%

\section{Continuous-time cocycles} \label{s.proofs_C}

In this section we prove the continuous-time \cref{t.slow_C,t.closing_C}.

\subsection{Preliminaries}\label{ss.prelim_C}

\subsubsection{Cartan--Hadamard manifolds}\label{sss.cartan-h}

Assume $H$ is a Hilbert-manifold, that is, a 
separable $C^\infty$-manifold modeled on a separable real Hilbert space 
$(\mathcal{H},\langle \mathord{\cdot},\mathord{\cdot} \rangle)$. Fix 
a Riemannian metric on $H$. (See \cite{Lang} for the precise definition.)

If $H$ is complete, simply connected, and has nonpositive sectional curvature,
then $H$ is called a \emph{Cartan--Hadamard manifold}. In this case, 
the Cartan--Hadamard--McAlpin Theorem (see \cite[\S~IX.3]{Lang}) states that for each 
point $p\in H$, the exponential map $\exp_p \!: T_p H \to H$ is a diffeomorphism.

\begin{remark}\label{r.sectional}
Let $\kappa \le 0$.
If $H$ is complete, simply connected, and has sectional curvature $\le \kappa$ (resp. $\ge \kappa$) everywhere, 
then $H$ has curvature $\le \kappa$ (resp. $\ge \kappa$) in the sense of Alexandrov;
see \cite[Chap.~6]{BBI}. \closeremark 
\end{remark}

\subsubsection[Killing fields and symmetric manifolds]
{Killing fields and symmetric Cartan--Hadamard manifolds}\label{sss.killing}

Here we recall some general facts about symmetric manifolds and Killing fields; 
more information can be found in \cite[Ch.~XIII]{Lang}.

\medskip

If $H$ is Hilbert-manifold, a \emph{Killing field} is a vector field 
that generates a (globally defined) flow of isometries.
Then the flow also preserves the Riemannian connection. 
On the space $\Kill(H)$ of these fields, we consider the first-countable topology 
for which the convergence of sequences is uniform convergence on bounded subsets. 

We endow $C(\Omega, \Kill(H))$ with the compact-open topology. Then a sequence 
$(\mathfrak{a}_n)$ in $C(\Omega,\Kill(H))$ converges to $\mathfrak{a}$ iff for every 
bounded set $B \subset H$, $\|\mathfrak{a}_n(\omega)(p) - \mathfrak{a}(\omega)(p)\|$ 
converges to $0$ uniformly with respect to $(\omega,p) \in \Omega \times B$.

\medskip

Now let $H$ be a Cartan--Hadamard manifold. 
Suppose it is symmetric in the sense of \cref{sss.symm}.\footnote{It is 
easy to check that this agrees with the definition from \cite[p.~359]{Lang}.}
If $v_0 \in T_{p_0}H$ is a nonzero vector, let $\alpha \!:  \R \to H$ be the 
geodesic such that $\alpha(0) = p_0$, $\alpha'(0)=v_0$. Consider the transvection 
$$
\tau_{\alpha,s} := \sigma_{\alpha(s/2)} \circ \sigma_{\alpha(0)}.
$$
Then $\tau_{\alpha,s}$ is a flow of isometries called the \emph{translation 
along $\alpha$}. More precisely, we have  $\tau_{\alpha, s} (\alpha(t)) = \alpha(t+s)$.
Moreover, the derivative 
$$
T_{\alpha(t)} \tau_{\alpha,s}  \!:  T_{\alpha(t)} H \to T_{\alpha(t+s)} H
$$
is the parallel transport along the geodesic $\alpha$. Let $\xi_{v_0}$ 
denote the Killing field that generates the flow $\tau_{\alpha,s}$. 
(For $v_0 = 0$, we define $\xi_{v_0} \equiv 0$.) Then the map 
$v_0 \in TH \mapsto \xi_{v_0} \in \Kill(H)$ is continuous. 

As it is customary, we denote by $\fm_{p_0}$ the set of Killing fields 
$\xi_{v_0}$, where $v_0 \in T_{p_0}H$. This is a vector space, 
and it can also be expressed as 
$$
\fm_{p_0} = \big\{ \xi \in \Kill(H) \, \!:  
\nabla_\zeta \xi (p) =0 \text{ for all vector fields } \zeta \big \},
$$
where $\nabla$ denotes covariant derivative.

\subsubsection{Infinitesimal displacement estimates}

The following lemma is the infinitesimal counterpart of 
\cref{l.displ_estimate_proper,l.displ_estimate_nonproper}:

\begin{lemma}\label{l.displ_estimate_C}
Let $H$ be a symmetric Cartan--Hadamard manifold.
There is a non-decreasing function $f \!: \R_+\to\R_+$ with $f(0)=1$ such that
$$
\|v_0\| \stackrel{\text{\rm (I)}}{\le} 
\| \xi_{v_0} (p) \|  \stackrel{\text{\rm (II)}}{\le} 
f(d(p,p_0)) \|v_0\| 
$$
for all $p_0$, $p \in H$, $v_0 \in T_{p_0}H$.
\end{lemma}

Inequality (I) above is related to nonpositive curvature; let us prove it first:

\begin{proof}[Proof of part (I) in \cref{l.displ_estimate_C}]
Fix $p_0, p$ in $H$, $v_0 \in T_{p_0} H$.
Assume $v_0 \neq 0$, otherwise there is nothing to prove.
Let $\beta$ be a unit-speed geodesic joining $p_0$ and $p$.
Let $\eta(t) := \xi_{v_0}(\beta(t))$.
Then (see \cite[Prop.~2.2, Ch~XIII]{Lang}) $\eta$ is a Jacobi field over the geodesic $\beta$.
By \cite[Prop.~5.6, Ch.~XIII]{Lang}, 
we have $\nabla_{\beta'} \eta (0) = 0$.
Let $g(t) := \| \eta(t) \|^2$.
By nonpositive curvature, this function is convex; 
see \cite[Lemma~1.1, Ch.~X]{Lang}. 
The same lemma also says that 
$g' = 2 \langle \nabla_{\beta'} \eta, \eta \rangle$,
which vanishes at $t=0$.
It follows that $g(t) \ge g(0)$ for all $t \in \R$.
In particular, $\| \xi_{v_0}(p) \|^2 \ge \| \xi_{v_0}(p_0) \|^2 = \|v_0\|^2$,
thus completing the proof of inequality~(I).
\end{proof}

If $H$ is finite-dimensional then the existence of a function with property (II)
in \cref{l.displ_estimate_C} is nearly trivial, and does not rely on nonpositive curvature:

\begin{proof}[Proof of part (II) in \cref{l.displ_estimate_C} assuming $\dim H < \infty$]
Consider
$$
\tilde f(p_0, \ell) := \sup \big\{ \|\xi_{v_0}(p)\| \, \!:  
p \in H \text{ with } d(p,p_0)\le \ell, \  v_0 \in T_{p_0}H   \text{ with }  \|v_0\| = 1 \big\},
$$
which is finite by compactness.
Since $\Isom(H)$ acts transitively on $H$, the value $\tilde f(p_0, \ell)$ actually does 
not depend on $p_0$, and so defines a function $f(p_0)$ with the required properties.
\end{proof}

The proof of (II) in the infinite-dimensional case requires geometric arguments
and is given in the Appendix~\ref{s.killing}.

\subsubsection{Infinitesimal uniform homogeneity}\label{sss.homog_inf}

The following is an infinitesimal version of the macroscopic uniform homogeneity,  
i.e., \cref{l.homog_macro}. It basically says that we can move any point $p$ in 
any desired direction $w$ by an infinitesimal isometry (Killing field), and 
these fields can be chosen so that they converge uniformly (with respect 
to $p$) in bounded sets to zero as $\|w\| \to 0$.

\begin{lemma}[Infinitesimal uniform homogeneity]\label{l.homog_inf}
Let $p_0 \in H$.
There is a continuous map 
\begin{align*}
K  \!: TH &\to     \Kill (H) \\
          w  &\mapsto K_{w}
\end{align*}
with the following properties:
\begin{enumerate}
\item The vector field $K_w$ extends $w$, that is, if $p = \pi(w) \in H$ 
is the base point of $w$, then $K_{w}(p) = w$;
\item 
For any $q \in H$,
$$
\| K_{w}(q)\|  \le  f(d(q,p_0)) \| w \|,
$$
where $f$ is given by \cref{l.displ_estimate_C}.
\end{enumerate}
\end{lemma}

\begin{proof}[Proof of \cref{l.homog_inf}]
Let $p_0 \in H$ be fixed. 
For $p \in H$, consider the map
\begin{align*}
L_{p}  \!:  T_{p_0} H &\to     T_p H        \\
                  v_0 &\mapsto \xi_{v_0}(p).
\end{align*}
We list below some properties of $L_{p}$:
\begin{itemize}
\item It is linear; see \cite[p.~363]{Lang}.
\item It is continuous; see part (II) of  \cref{l.displ_estimate_C}.
\item It is one-to-one, and the inverse (on the image) is continuous; see part (I) of  \cref{l.displ_estimate_C}.
\item It is onto; see \cref{l.onto} in Appendix~\ref{s.killing}. (In finite 
dimension, this part would of course be a trivial consequence of the others.)
\end{itemize}
Given $w \in T_p H$, define $K_{w} := \xi_{(L_{p})^{-1}(w)}$.
Using parts (II) and (I) of \cref{l.displ_estimate_C}, we have
\begin{equation*}
\| K_{w}(q)\|  \le  f(d(q,p_0)) \| K_{w}(p_0)\| \le  f(d(q,p_0)) \| w \|. \qedhere
\end{equation*}
\end{proof}

\begin{remark}
Although there is no apparent advantage in doing so, 
it is possible to give an alternative proof of \cref{l.homog_macro} using \cref{l.homog_inf}:
Given two points $p$ and $q$, join them by a geodesic $\gamma \!: [0,\ell] \to H$,
and integrate the time-varying Killing field $K_{\gamma'(t)}$
(where $K$ is given by \cref{l.homog_inf}) to get the map $J(p,q)$.~\closeremark 
\end{remark}

\subsection[Proof of \cref{t.slow_C}]{Existence of sections of nearly minimal 
speed: proof of \cref{t.slow_C}}\label{ss.proof_slow_C}

For each $p\in H$, we let $\delta_p$ or $\delta(p)$ denote the Dirac measure at the point $p$. 
If $\gamma \!: [a,b] \to H$ is a curve, we denote by $\int_a^b \delta(\gamma(t)) \,\mathrm{d}t$
the measure on $H$ obtained by pushing-forward by $\gamma$ the Lebesgue measure on $[a,b]$.

In the proof of \cref{t.slow_C}, we will need the following technical result, 
whose proof is given in the Appendix~\ref{s.diff_cartan}.

\begin{lemma}[Differentiability of the Cartan barycenter]\label{l.diffbar}
Let $I\subset \R$ be an open interval and let 
$h \!: I \times [0,T] \to H$ be a continuous mapping 
that is continuously differentiable with respect to the first variable.
Then the map $\bar{h} \!:  I \to H$ defined by
\begin{equation}\label{e.hbar}
\bar{h}(t) = \bary \left( \frac{1}{T} \int_0^T \delta_{h(t,s)} \, \mathrm{d}s \right)
\end{equation}
(where $\bary$ denotes the Cartan barycenter)
is continuously differentiable.
\end{lemma}

\begin{proof}[Proof of \cref{t.slow_C}]
Fix any $p_0 \in H$.
For $T>0$, let
\begin{equation}\label{e.magic3}
\phi_T (\omega) = 
\bary \left( \frac{1}{T} \int_0^T \delta\left( [A^{(t)}(\omega)]^{-1}  
p_0 \right)  \,\mathrm{d}t  \right)
\end{equation}
(Compare with \eqref{e.magic2}.)

\begin{claim}
The function $\phi_T \!:  \Omega \to H$ is differentiable with respect to the semiflow.
\end{claim}

\begin{proof}[Proof of the claim]
We have 
$$
\phi_T (F^t \omega) = 
\bary \left( \frac{1}{T} \int_0^T \delta\left( [A^{(s)}(F^t \omega)]^{-1}  p_0 \right)  \,\mathrm{d}s  \right) \, .
$$
By \cref{l.diffbar}, to show that the map $\phi_T (F^t \omega)$
is continuously differentiable with respect to $t$, it suffices to check 
that the map $(s,t) \mapsto [A^{(s)}(F^t \omega)]^{-1}  p_0 \in H$ is 
continuous and continuously differentiable with respect to $t$. But 
these properties follow from the cocycle identity \eqref{e.cocycle_C} 
and the regularity assumptions \eqref{e.regularity}.
\end{proof}

Next, we want to estimate the distance:
\begin{equation}\label{e.to_estimate}
d \big(\phi_T (F^t \omega) , A^{(t)} (\omega)  \phi_T(\omega) \big)  =
d \big( \underbrace{[A^{(t)} (\omega)]^{-1} (\phi_T (F^t \omega))}_{(\star)}, \phi_T(\omega) \big).
\end{equation}
Assuming $t \in (0,T)$, we have
\begin{align*}
(\star) 
&= \bary  \frac{1}{T} \int_0^T \delta\left( [A^{(t)} (\omega)]^{-1} [A^{(s)}(F^t \omega)]^{-1}  p_0 \right)  \,\mathrm{d}s  \\
&= \bary  \frac{1}{T} \int_0^T \delta\left( [A^{(s+t)} (\omega)]^{-1}  p_0 \right)  \,\mathrm{d}s  \\
&= \bary  \frac{1}{T} \int_t^{T+t} \delta\left( [A^{(s)} (\omega)]^{-1}  p_0 \right)  \,\mathrm{d}s.
\end{align*}
Using the barycenter property \eqref{e.bar_lips_C},
we obtain that the distance \eqref{e.to_estimate} is at most
$$
\frac{t}{T} \sup \left\{ d\big([A^{(s)} (\omega)]^{-1}  p_0 ,[A^{(u)} (\omega)]^{-1}  p_0  \big) \, \!:  s \in [0,t], u \in [T,T+t] \right\} \, .
$$
Dividing by $t$ and making $t \to 0$, we obtain
$$
\| \phi_T'(\omega) - \mathfrak{a} (\omega)(\phi_T(\omega))\|
\le \frac{1}{T} \dist \big(A^{(T)}(\omega)  p_0, p_0 \big) \, .
$$
In particular,
$$
\speed (\phi_T) \le \sup_{\omega \in \Omega} 
\frac{1}{T} \dist \big(A^{(T)}(\omega)  p_0, p_0 \big).
$$
The theorem follows by taking $\phi=\phi_N$ with sufficiently large $N$.
\end{proof}

\subsection[Proof of \cref{t.closing_C}]
{Creating invariant sections: proof of 
\cref{t.closing_C}}\label{ss.proof_closing_C}

\begin{proof}[Proof of \cref{t.closing_C}]
By \cref{t.slow_C}, there exists a family of sections 
$\phi_T$ such that $\lim_{T \to \infty} \speed(\phi_T) = 0$.
Let $K(w) = K_{w}$ be the map given by \cref{l.homog_inf}.
Define
\begin{equation}\label{e.closing_C}
\tilde{\mathfrak{a}}_T (\omega) = \mathfrak{a}(\omega) + K \big( \phi_T'(\omega) - \mathfrak{a}(\omega)(\phi_T(\omega)) \big).
\end{equation}
(Compare with \eqref{e.closing_D}.)
Then:
\begin{itemize}
\item $\tilde{\mathfrak{a}}_T (\omega) (\phi_T(\omega)) = \phi'_T(\omega)$, that is,
$\phi_T$ is an invariant section for the cocycle generated by $\tilde{\mathfrak{a}}_T$;
\item for each bounded subset $B$ of $H$,
the sequence $\tilde{\mathfrak{a}}_T (\omega) (p)$ converges to $\mathfrak{a}(\omega)(p)$ 
uniformly with respect to $(\omega,p) \in \Omega \times B$.
\end{itemize}
Thus the theorem is proved.
\end{proof}

\begin{remark}\label{r.fibered_C}
Similarly to \cref{r.fibered_D}, it should be possible to state and prove fibered versions of \cref{t.slow_C,t.closing_C}, but we have not checked that.
It seems to be necessary to use a connection on the bundle 
in order to define the speed of a section.~\closeremark
\end{remark}

\appendix

\section{Appendix: The displacement estimate for non-proper spaces} \label{s.alexandrov}

In this appendix we prove \cref{l.displ_estimate_nonproper}. 
We will actually obtain an explicit formula for the function $f$.
Some preliminaries are needed.

\subsection{Angles and more comparisons}

\begin{proposition}\label{p.angle}
Let $H$ be a space of curvature $\le \kappa$ (resp.\ $\ge \kappa$), 
with $\kappa \leq 0$. Let $\gamma_1$, $\gamma_2$ be two geodesics 
such that $\gamma_1(0) = \gamma_2 (0) = p_0$. For each $t > 0$, $s>0$, 
consider the triangle $\triang(p_0,\gamma_1(t), \gamma_2(s))$,
and let $\triang(\tilde p_0,\tilde p_{1,t}, \tilde p_{2,s})$ 
be an SSS-comparison triangle in 
$M_\kappa$.
Let $\theta_\kappa(t,s)$ be the angle at the vertex $\tilde p_0$. Then 
the function $\theta_\kappa(t,s)$ is monotonically nondecreasing 
(resp.\ nonincreasing) with respect to each variable.
\end{proposition}

\begin{proof}
See \cite{ABN}. 
\end{proof}

In particular, the limit of $\theta_\kappa(t,s)$ as $(t,s) \to (0,0)$ exists.
(Notice that the limit actually does not depend on $\kappa$.) It is 
called the \emph{angle} between $\gamma_1$ and $\gamma_2$ at $p_0$. 

An immediate consequence of \cref{p.angle} is that if $H$ has 
curvature $\le \kappa$ (resp.\ $\ge \kappa$), with $\kappa \leq 0$, 
then the angles of an SSS-comparison triangle in $M_\kappa$ are smaller 
(resp.\ greater) than or equal to the corresponding angles for the triangle in $H$.

If $\triang(p_1,p_2,p_3)$ is a triangle in $H$ and 
$\triang(\tilde p_1, \tilde p_2, \tilde p_3)$ is a triangle in $M_\kappa$
such that the angles at the vertices $p_1$ and $\tilde p_1$ are equal
and the corresponding sides at these vertices have equal lengths
(i.e., $\dist_\kappa(\tilde p_1, \tilde p_j) = \dist(p_1, p_j)$ for $j=2$, $3$), 
then we say that $\triang(\tilde p_1, \tilde p_2, \tilde p_3)$ is an  
\emph{SAS-comparison triangle}\footnote{SAS stands for side-angle-side.} 
for $\triang(p_1,p_2,p_3)$.

\begin{lemma}\label{l.SAS}
Let $H$ be a uniquely geodesic space of curvature $\ge \kappa$ 
(where $\kappa \le 0$). Let $\triang(p_1,p_2,p_3)$ be a triangle in $H$,
and let $\triang(\tilde p_1, \tilde p_2, \tilde p_3)$ be 
an SAS-comparison triangle in $M_\kappa$
(with equal angles at vertices $p_1$ and $\tilde p_1$).
Then:
\begin{enumerate}
\item\label{i.side} 
The side $p_2 p_3$ is shorter than or has the same length 
as the side $\tilde p_2 \tilde p_3$.
\item\label{i.angle} 
If the angles at $p_2$ and $p_3$ are less than $\pi/2$, 
then the angle at $p_2$ (resp.\ $p_3$) is larger than 
or equal to the angle at $\tilde p_2$ (resp.\ $\tilde p_3$).
\end{enumerate}
\end{lemma}

\begin{proof}
Consider a triangle $\triang(p_1,p_2,p_3)$ in $H$ (which we assume has curvature $\ge \kappa$).
In $M_\kappa$, we take an SAS-comparison triangle $\triang(\tilde p_1, \tilde p_2, \tilde p_3)$ 
(so that the angles at $p_1$ and $\tilde p_1$ are equal)
and an SSS-comparison triangle $\triang (q_1, q_2, q_3)$.
Then the angle at $q_i$ is less than the angle at $p_i$.
Thus, to complete the proof of the lemma, we need 
the following facts about plane hyperbolic geometry,
whose proof we will leave as an exercise.

\begin{claim}  
Suppose $\triang(q_1, q_2, q_3)$ and $\triang(\tilde p_1, \tilde p_2, \tilde p_3)$ 
are triangles in $M_\kappa$ (where $\kappa \le 0$) so that the angle at vertex 
$\tilde p_1$ is bigger than the angle at vertex $q_1$, and the adjacent sides are equal. 
Then:
\begin{enumerate}
\item The side $\tilde p_2 \tilde p_3$ is bigger than the side $q_2 q_3$.

\item If the angles at $\tilde p_2$ and $\tilde p_3$ are both less than $\pi/2$, 
then they are smaller than the angles at $q_2$ and $q_3$, respectively. 

\end{enumerate}
\end{claim}
\end{proof}

\subsection{The displacement estimate}

The following is a more precise version of \cref{l.displ_estimate_nonproper}:

\begin{lemma}\label{l.displ_estimate_again}
Assume that $H$ is a geodesic symmetric space 
of bounded nonpositive curvature in the sense of Alexandrov.
Let $\kappa<0$ be a lower bound for the curvature. Let $J$ be 
a transvection that translates a geodesic $\gamma$ by length $b$.
Given $q \in H$, let $s = \dist(J(q),q)$, $\ell = \dist(q,\gamma)$, and $\lambda=\sqrt{-\kappa}$.
Then
\begin{equation}\label{e.complicated}
\cosh \lambda s \le 
\cosh^2 \frac{\lambda b}{2} \cosh^2 2 \lambda \ell + 
\sinh^2 \frac{\lambda b}{2} \cosh 2\lambda\ell -
\cosh \frac{\lambda b}{2}  \sinh^2 2 \lambda\ell.
\end{equation}
\end{lemma}

\begin{proof}
Multiplying the metric by a constant, we can assume that $\kappa=-1$, i.e., $\lambda=1$.

Let $J$ be a transvection that translates a geodesic $\gamma$ by length $b$, and let $q \in H$.
Let $s = \dist(J(q),q)$ and $\ell = \dist(q,\gamma)$.
We can assume that $b > 0$ (because for $b = 0$ formula \eqref{e.complicated} means $s=0$)
and $\ell > 0$ (because for $\ell = 0$ formula \eqref{e.complicated} means $s \le b$).

Let $p_0$ be the point in $\gamma$ which is closest to $q$.
Let $p_1$ be the midpoint of $p_0$ and $J(p_0)$;
then $J = \sigma_{p_1} \circ \sigma_{p_0}$.
Consider the triangle with vertices $p_0$, $p_1$, and $\sigma_{p_0}(q)$.
The angle at vertex $p_0$ is $\pi/2$;
let $\alpha$ and $\beta$ the angles at vertices $p_1$ and $\sigma_{p_0}(q)$,
respectively.
See \cref{f.transvection}.

\psfrag{p0}[r][r]{$p_0$}
\psfrag{q}[r][r]{$q$}
\psfrag{z}[r][r]{$\sigma_{p_0}(q)$}
\psfrag{Jq}[l][l]{$J(q)$}
\psfrag{Jp0}[l][l]{$J(p_0)$}
\psfrag{w}[l][l]{$J(\sigma_{p_0}(q)) = \sigma_{p_1}(q)$}	
\psfrag{p1}[l][l]{$p_1$}
\psfrag{s}[c][c]{$s$}
\psfrag{al}[r][r]{$\alpha$}
\psfrag{be}[c][c]{$\beta$}
\psfrag{l}[r][r]{$\ell$}
\psfrag{b/2}[c][c]{$b/2$}
\psfrag{d1}[l][l]{$d/2$}
\psfrag{d2}[r][r]{$d/2$}
\begin{figure}[htb]
\begin{center}
\includegraphics[width=.55\textwidth]{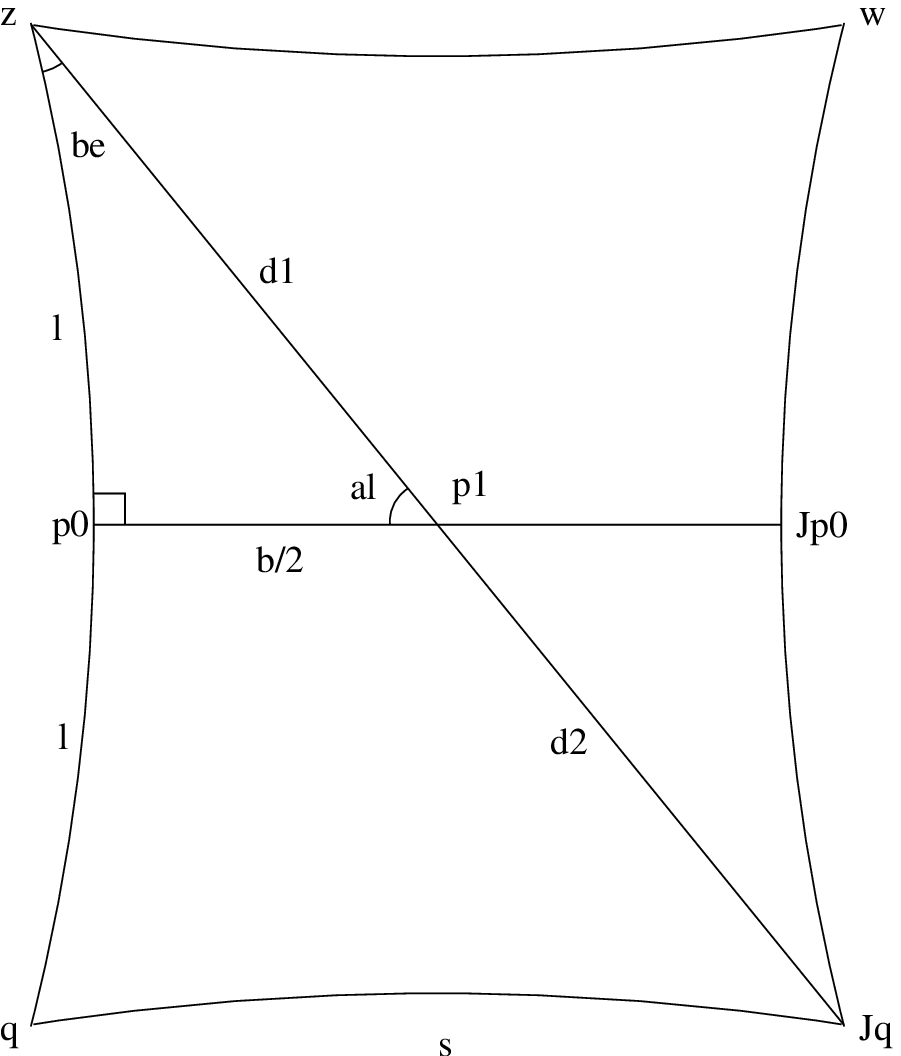}
\end{center}
\caption{Proof of \cref{l.displ_estimate_again}}\label{f.transvection}
\end{figure}

Now consider an SAS-comparison triangle in the hyperbolic plane $M_{-1}$,
more precisely a triangle in $M_{-1}$ with two sides $b/2$ and $\ell$
and angle between them $\pi/2$. Let $\tilde \alpha$, $\tilde \beta$ 
be the respective angles, and let $\tilde{d}/2$ be the third side.
By \cref{l.SAS}, we have
$$
d \le \tilde d, \quad
\alpha \ge \tilde \alpha, \quad
\beta \ge \tilde \beta.
$$
By the law of cosines in $M_{-1}$, we have:
\begin{equation}\label{e.cosh}
\cosh\frac{\tilde d}{2} = \cosh \frac{b}{2} \cosh \ell \, .
\end{equation}
By the law of sines in $M_{-1}$, we have:
\begin{equation}\label{e.sinh}
\sinh \frac{\tilde{d}}{2} = \frac{\sinh \ell}{\sin \tilde \alpha} \, . 
\end{equation}

By the law of cosines in $H$ (an inequality which comes 
automatically from the curvature lower bound), we have:
\begin{alignat*}{2}
\cosh s 
&\le \cosh d \cosh 2 \ell - \sinh d \sinh 2\ell \cos \beta 
&\hspace{-1cm}&\text{(law of cosines)} \\
&\le \cosh d \cosh 2 \ell - \sinh d \sinh 2\ell \sin \alpha 
&\hspace{-1cm}&\text{(since $\alpha + \beta \le \pi/2$)} \\
&\le \cosh d \cosh 2 \ell - \sinh d \sinh 2\ell \sin \tilde \alpha 
&\hspace{-1cm}&\text{(since $\tilde \alpha \le \alpha \le \pi/2$)} \\
&\le \cosh \tilde d \cosh 2 \ell - \sinh \tilde d \sinh 2\ell \sin \tilde \alpha 
&\hspace{-1cm}&\text{(since $\tilde d \ge d \ge 2\ell$)} \\
&= \left(2\cosh^2 \frac{\tilde d}{2} - 1\right) 
\cosh 2 \ell - 2 \sinh \frac{\tilde d}{2} \cosh\frac{\tilde d}{2} \sinh 2\ell \sin \tilde \alpha.
&\hspace{-1cm}& 
\end{alignat*}
Substituting \eqref{e.cosh} and \eqref{e.sinh} and manipulating,
we obtain \eqref{e.complicated}.
\end{proof}

\begin{remark}
\cref{f.transvection} is not necessarily contained 
in a ``two-dimensional'' totally geodesic subspace. 
If this were the case, it is possible to show that the 
following improved version of \eqref{e.complicated} holds:
$\cosh \lambda s \le  \cosh \lambda b \cosh^2 \lambda \ell - \sinh^2 \lambda \ell$.
Moreover, if $H=M_{-\lambda^2}$, then this becomes an equality, expressing the 
summit $s$ of a Saccheri quadrilateral as a function of the legs $\ell$ and 
the base $b$; see~\cite[p.~104]{BK}.~\closeremark 
\end{remark}

\section{Appendix: Some lemmas on Killing fields} \label{s.killing} 

In this appendix, we complete the proofs of \cref{l.displ_estimate_C,l.homog_inf},
which were proven in \cref{ss.prelim_C} only in the finite dimensional case.

\begin{remark}\label{r.symm_bound}
The sectional curvatures of a symmetric Cartan--Hadamard manifold are bounded from below.
Indeed, since isometries act transitively, it is sufficient to 
show that sectional curvatures are bounded at each point; but this 
follows directly from the boundedness of the Riemann tensor.\closeremark
\end{remark}

\begin{proof}[Proof of part (II) of \cref{l.displ_estimate_C}]
Let $\kappa$ be the infimum of the sectional curvature of $H$,
which is finite by the previous remark.
We will show that (II) holds with 
$$
f(\ell) :=  \cosh \left( 2 \sqrt{-\kappa} \cdot \ell \right).
$$

Fix $p_0, p$ in $H$, $v_0 \in T_{p_0} H$.
Assume $v_0 \neq 0$, otherwise there is nothing to prove.
Let $\alpha$ be the geodesic passing through $p_0$ with velocity $v_0$.
Let $\beta$ be a unit-speed geodesic joining $p_0$ and $p$.
Let $\eta(t) = \xi_{v_0}(\beta(t))$. Then 
(see \cite[Prop.~2.2, Ch~XIII]{Lang}) 
$\eta$ is a Jacobi field over the geodesic $\beta$.

In view of the Rauch--Berger comparison theorem (see \cite{Biliotti, CE}),
in order to show that
\begin{equation}\label{e.whatever}
\|\eta(t) \| \le f(d( \beta(t) , \alpha)) \|v_0\| \le f(t) \|v_0\|,
\end{equation}
we need only to consider the case where $H$ is the hyperbolic plane of 
constant curvature $\kappa$. But then it is a simple calculation; actually, 
in this case, the first inequality in \eqref{e.whatever} becomes an equality.
\end{proof}

\begin{lemma}\label{l.onto} 
Let $H$ be a symmetric Cartan--Hadamard manifold.
For any $p_0, p$ in $H$, $v \in T_{p} H$, there exists $\chi \in \fm_{p_0}$ such that
$\chi(p)=v$.
\end{lemma}

\begin{proof}
Let $\ell = d(p,p_0)$.
Assume $\ell>0$, otherwise the claim is trivial.
Let $\beta \!:  \R \to H$ be the geodesic such that $\beta(0) = p_0$ and $\beta(\ell) = p$.
Let $q = \beta(-\ell)$.
There is a Jacobi vector field $\eta$ over $\beta$ such that 
$$
\eta(\ell) = v , \quad \eta(-\ell) = 0.
$$
(The existence of $\eta$ follows from the Cartan--Hadamard--McAlpin theorem; 
see \cite[\S~IX.3]{Lang} and \cite[Thrm.~IX.3.1]{Lang}.) We claim that 
$\xi_{2 \eta(0)}$ is the sought-after Killing field.

Let $\sigma = \sigma_{p_0}$ be the symmetry at $p_0$; then $\sigma(\beta(t)) = \beta(-t)$.
Let $\zeta$ be the Jacobi field over $\beta$ obtained by pushing-forward $\eta$ by $\sigma$, that is,
$$
\zeta(t) = T\sigma(\beta(-t)) \cdot \eta(-t).
$$
Consider the Jacobi field $\chi = \eta - \zeta$. (See \cref{f.jacobis}.)
\psfrag{0}[l][l]{$p_0$}
\psfrag{p}[r][r]{$p$}
\psfrag{s}[l][l]{$\sigma_{p_0}(p)$}
\psfrag{e}[r][r]{$\eta$}
\psfrag{z}[c][c]{$\zeta$}
\psfrag{x}[c][c]{$\chi$}
\begin{figure}[htb]
\begin{center}
\includegraphics[width=.55\textwidth]{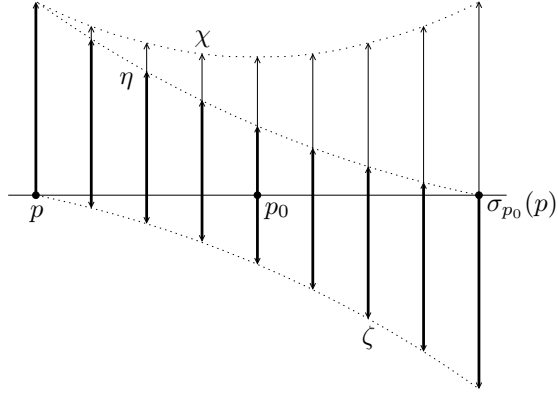}
\end{center}
\caption{The Jacobi field $\eta$, its reflection $\zeta$, and the field $\chi = \eta - \zeta$.}
\label{f.jacobis}
\end{figure}

The isometry $\sigma$ preserves covariant derivative, hence  
$$
\nabla_{\beta'}\zeta(t) = - T\sigma(\beta(-t)) \cdot \nabla_{\beta'}\eta(-t).
$$
Taking $t=0$, we get
$\nabla_{\beta'} \chi (0)= 0$.
As a consequence (use \cite[Prop.~XIII.5.6]{Lang}), the Jacobi field $\chi$ over $\beta$
can be extended to a Killing field $\chi \in \fm_{p_0}$. 
This is exactly $\xi_{\chi(0)} = \xi_{2\eta(0)}$.
\end{proof}

\section{Appendix: Differentiability of the Cartan barycenter}\label{s.diff_cartan}

In this appendix, we give the proof of \cref{l.diffbar}.

Assume $H$ is a Cartan--Hadamard manifold.
Let $\mu$ be a probability measure of bounded support (p.o.b.s.) on $H$,
and let $f=f_\mu$ be given by \eqref{e.sum_squares}.

\begin{lemma}
The gradient vector field\footnote{Defined by the relation 
$Tf_p (v) = \langle \grad f (p), v \rangle$.} of $f$
is given by:
$$
\grad f (p) = - \int_H \exp_p^{-1}(q)  \, \mathrm{d}\mu(q) \, .
$$
\end{lemma}

\begin{proof}
See \cite[p.~132]{BK}.
\end{proof}

\begin{lemma}\label{l.invert}
For each $p$, the linear map 
$$L \!:  v \in T_p M \mapsto \nabla_v \grad f (p)  \in T_p M$$ 
(given by covariant derivative) 
is bounded and has a bounded inverse. 
\end{lemma}

\begin{proof}
By the previous lemma, 
$\grad f = \int_H \xi_q \, \mathrm{d}\mu(q)$, where $\xi_q$ is the vector field
$\xi_q (p) = - \exp_p^{-1}(q)$.
For each $q$ and $p$, 
the linear map $v \in T_p M \mapsto \nabla_v \xi_q (p)  \in T_p M$
is symmetric and $\ge \Id$; see \cite[p.~172, 188]{Karcher}.
Integrating over $q$, we conclude that $L$ is self-adjoint and $\ge \Id$.
Now we need the following fact:

\begin{claim} 
Let $L$ be a bounded self-adjoint operator 
on a real Hilbert space such that $L \ge \Id$.
Then $L$ is invertible, with a bounded inverse.
\end{claim}

\begin{proof}[Proof of the Claim]
Let $c>0$ and estimate
$$
\|(\Id- cL)  v \|^2 
= \langle v, v \rangle - 2 c \langle Lv, v \rangle + c^2 \langle Lv, Lv \rangle
\le \left( 1 - 2c  + c^2 \|L\|^ 2\right) \|v\|^2 .
$$
Thus if $c$ is small enough, then $\|\Id - c L \| < 1$.
In particular, $cL$ is invertible, and so is $L$. 
\end{proof}
The lemma follows.
\end{proof}

\begin{proof}[Proof of \cref{l.diffbar}]
Define a one-parameter family of vector fields $\xi_t$ ($t \in I$)
on $H$ by 
$$
\xi_t(p) = - \frac{1}{T} \int_0^T \exp_p^{-1}(h(t,s))  \, \mathrm{d}s \, .
$$
Then $\bar{h}(t)$ is the unique solution of $\xi_t(\bar{h}(t)) = 0$.
To deduce differentiability of $\bar{h}$ from the Implicit Function Theorem,
it is sufficient to check that for each $t \in H$ and $p \in H$,
the linear map $v \in T_p M \mapsto \nabla_v \xi_t (p) \in T_p M$ (given by covariant derivative)
is invertible (with a bounded inverse) -- see \cite[p.~143]{BK} for details. 
But this was proved in \cref{l.invert} above.
\end{proof}

\begin{remark}
Another approach to the differentiability of the barycenter is to consider 
barycenters in the tangent bundle $TH$; see \cite{ArLi}.~\closeremark 
\end{remark}


\begin{small}

\phantomsection 

\vspace{0.5cm}

\noindent
\begin{minipage}[t]{.4\linewidth}
Jairo Bochi


PUC--Rio

Rua Marqu\^es de S.~Vicente, 225

Rio de Janeiro, Brazil

jairo@mat.puc-rio.br

\href{http://www.mat.puc-rio.br/~jairo}{www.mat.puc-rio.br/$\sim$jairo}
\end{minipage}
\hspace{.2\linewidth}
\begin{minipage}[t]{.4\linewidth}
Andr\'es Navas

Universidad de Santiago

Alameda 3363, Estaci\'on Central

Santiago, Chile

andres.navas@usach.cl
\end{minipage}

\end{small}

\end{document}